\documentclass[11pt]{amsart}
\usepackage{graphicx, mathabx, amssymb,amsfonts,amsmath,amsthm,newlfont}
\usepackage{epsfig,url}
\usepackage{enumerate,enumitem}
\usepackage[colorlinks=true,linkcolor=red,citecolor=blue]{hyperref}
\usepackage[dvipsnames]{xcolor}
\usepackage{color}

\usepackage[noabbrev,capitalize]{cleveref}

\usepackage[all,2cell]{xy} \UseAllTwocells \SilentMatrices

\usepackage{pstricks,pst-node,pst-tree}

\calclayout

\newtheorem{thm}{Theorem}[section]
\newtheorem{corollary}[thm]{Corollary}
\newtheorem{lemma}[thm]{Lemma}
\newtheorem{proposition}[thm]{Proposition}

\theoremstyle{definition}
\newtheorem{definition}[thm]{Definition}

\newtheorem{example}[thm]{Example}

\theoremstyle{remark}
\newtheorem{rem}[thm]{Remark}

\numberwithin{equation}{section}

\newcommand{\set}[1]{\left\{#1\right\}}

\newcommand{\R}{\mathbb R}

\newcommand{\ari}{\mathrm{ar}} 

\newcommand{\eqdef}{\mbox{\,\raisebox{0.2ex}{\scriptsize\ensuremath{\mathrm:}}\ensuremath{=}\,}} 

\newcommand{\inc}{\subseteq}
 \newcommand{\incs}{\subsetneq}
\newcommand{\union}{\cup}
\newcommand{\Union}{\bigcup}	

\newcommand{\setc}[2]{\set{#1 \mid #2}}

\newcommand{\hyper}[1]{{\mathbb #1}}	
\newcommand{\restrH}[2]{\hyper{#1}\backslash #2}

\definecolor{darkblue}{rgb}{0,0,0.7} 
\newcommand{\darkblue}{\color{darkblue}} 
\newcommand{\defn}[1]{{\darkblue \emph{#1}}}

\newcommand{\occ}[2]{#1/#2}
\newcommand{\recrestr}[2]{#1_{\cap #2}}
\newcommand{\xyz}[3]{#1\stackrel{#2}{\rightsquigarrow}#3}
\newcommand{\pre}[2]{<_{#1,#2}}

\newcommand{\calH}{\mathcal{H}}
\newcommand{\calT}{\mathcal{T}}

\DeclareMathOperator{\Sat}{Sat}
\DeclareMathOperator{\supp}{supp}
\DeclareMathOperator{\Ter}{Ter}
\DeclareMathOperator{\outsort}{out}
\DeclareMathOperator{\insort}{in}
\DeclareMathOperator{\var}{var}

\usepackage{graphicx,calc}
\newlength\myheight
\newlength\mydepth
\settototalheight\myheight{Xygp}
\usepackage{tikz}
\usepackage{tikz-cd}
\usepackage{pgfplots}
\usepackage{pgfplotstable}
\tikzset{math3d/.style=
    {x= {(-0.353cm,-0.353cm)}, z={(0cm,1cm)},y={(1cm,0cm)}}}
\tikzset{JLL3d/.style=
    {x= {(0.4cm,-0.2cm)}, z={(0cm,1cm)},y={(-1cm,0cm)}}}
\usetikzlibrary{calc}
\usetikzlibrary{shapes,shapes.geometric,fit,positioning,calc,matrix}
\tikzset{
  optree/.style={scale=.5,thick,grow'=up,level distance=10mm,inner sep=1pt},
  comp/.style={draw=none,circle,fill,line width=0,inner sep=0pt},
  dot/.style={draw,circle,fill,inner sep=0pt,minimum width=3pt},
  circ/.style={draw,circle,inner sep=1pt,minimum width=4mm},
  emptycirc/.style={draw,circle,inner sep=1pt,minimum width=2mm},
  root/.style={level distance=10mm,inner sep=1pt},
  leaf/.style={draw=none,circle,fill,line width=0,inner sep=0pt},
  nodot/.style={draw,circle,inner sep=1pt},
}

\pgfplotsset{compat=1.12}

\title{Term rewriting on nestohedra}

\author{Pierre-Louis Curien}
\address{IRIF, Universit\'e Paris Diderot and $Picube$ team, Inria, France.}
\email{curien@irif.fr}

\author{Guillaume Laplante-Anfossi}
\address{School of Mathematics and Statistics, The University of Melbourne, Victoria, Australia.}
\email{guillaume.laplanteanfossi@unimelb.edu.au}

\date{\today}



\thanks{The second author was supported by the Australian Research Council Future Fellowship FT210100256 and the Andrew Sisson Fund.}

\begin{document}

\begin{abstract}
We define term rewriting systems on the vertices and faces of nestohedra, and show that the former are  
confluent and terminating. 
While the associated posets on vertices generalize Barnard--McConville's flip order for graph-associahedra, the preorders on faces generalize the facial weak order for permutahedra and the generalized Tamari order for associahedra. 
Moreover, we define and study contextual families of nestohedra, whose local confluence diagrams satisfy a certain uniformity condition. 
Among them are associahedra and operahedra, whose associated proofs of confluence for their rewriting systems reproduce proofs of categorical coherence theorems for monoidal categories and categorified operads.
\end{abstract}

\maketitle

\setcounter{tocdepth}{1}


\section*{Introduction} 
\label{s:introduction}

\subsection*{From rewriting to coherence}
In his seminal notes~\cite{Huet-notes-cat} for a graduate course at Université Paris 7, Gérard Huet explained Mac Lane's proof of the coherence theorem for monoidal categories through the lenses of equational reasoning and  term rewriting theory. 
Huet remarked that instantiations in context of Mac Lane's pentagon can be read as local confluence diagrams. Iterated tensor bifunctors  are represented as terms over the signature on a single operation $\otimes$ of arity 2, the associator gives rise to the single rewriting rule $(X\otimes Y)\otimes Z\to X\otimes(Y\otimes Z)$ (more details are provided in~\cref{ss:coherence}), and
\begin{enumerate}
\item proving  the coherence statement in the case of canonical natural transformations $\lambda:F\rightarrow G$, where $\lambda$ is defined using the associator only (and not its inverse) and where $G$ is a normal form for the above rewriting system, amounts to annotating the proof of Newman's lemma with  explicit names for the rewriting steps;
\item moreover, the proof of the general case of the coherence theorem resembles the proof of the Church-Rosser property, which states that if two terms $P,Q$ can be proved equal in the equational theory obtained by forgetting the orientation of the rewriting rules, then there is some $N$ such that $P\rightarrow\cdots\rightarrow N$ and
$Q\rightarrow\cdots\rightarrow N$.
\end{enumerate}
In addition, in order to check local confluence, it is enough to check local confluence of {\em critical pairs}, which are ``minimal'' situations in which $M\rightarrow P$ and $M\rightarrow Q$ in such a way that the respective subterms of $M$ to which the two reductions are applied overlap (see ~\cref{recollection-section}).

Huet notes that Mac Lane's pentagon expresses the local confluence of the unique critical pair of  the rewriting system given by the associator. The reader unfamiliar with the terminology of rewriting systems will find a brief hopefully self-contained introduction to rewriting in~\cref{recollection-section}.
Later on, we shall refer to the observations above (limited to the case of unit-free monoidal structure) as the original \defn{Huet correspondence}, see~\cref{Huet-correspondence} for a precise description.

\subsection*{Coherence and polytopes}

In a previous paper~\cite{CLA1}, we discussed combinatorial topological proofs of coherence theorems.
In particular, we gave an explicit  topological proof of Mac Lane's coherence theorem by using the fact that all diagrams involved live on the $2$-skeleton of a family of polytopes, the associahedra. 
Here, 
\begin{itemize}
\item[(0)] 0-cells correspond to functors, 
\item[(1)] paths in the 1-skeleton correspond to natural transformations,  
\item[(2)]pentagons as well as naturality and bifunctoriality squares correspond to 2-faces,  
\end{itemize}
and the coherence statement amounts to asking whether any two parallel cellular paths can be related by repeatedly replacing a portion of a path fitting on the boundary of a $2$-face by the complementary path on that same boundary.
In fact, our topological/combinatorial results can be applied to give ``one step proofs'' (quoting Kapranov~\cite{kapranov1993}) of a number of other categorical coherence theorems. 

\subsection*{Rewriting on nestohedra}

It is therefore natural to ask if we can factor Huet correspondence through the topological correspondence sketched above,  and extend it by associating  term rewriting systems to other kinds of polytopes, yielding  coherence results for different families of interest in a unified way.
In this paper, we give a positive answer to this question for the family of hypergraph polytopes, also known as nestohedra, associated with connected hypergraphs. 
We construct confluent  and terminating  term rewriting systems (\cref{thm:confluent}) on the vertices of hypergraph polytopes in such a way that edges are naturally oriented and feature rewriting steps.
We characterize the local confluence diagrams of their critical pairs as certain types of $2$-faces (\cref{critical-pairs-nestohedra}).
The rewriting steps on the vertices generalize Barnard--McConville's \emph{flip order} on the vertices of graph-associahedra~\cite{Barnard-McConville}, and are induced by an orientation vector (\cref{Tamari-orientation-vector}).
Meanwhile, we also define a rewriting system on all faces, which coincides with the \defn{facial weak order} in the case of permutahedra \cite{KrobLatapyNovelliPhanSchwer,PalaciosRonco,DermenjianHohlwegPilaud} and the \defn{generalized Tamari order} in the case of associahedra \cite{PalaciosRonco}, see \cref{rem:facialweak}.

\subsection*{Contextual hypergraphs}

We shall then specialize the discussion to \emph{contextual families} of hypergraphs (\cref{def:contextual-family}). 
Among these families, one finds the associahedra and the operahedra (\cref{thm:examples}), whose term rewriting systems provide, via the Huet correspondence, coherence theorems for monoidal categories and categorified operads, see \cref{rem:coherence} and \cref{ss:coherence}.  
The idea behind the condition satisfied by contextual nestohedra is to enforce the shape of local confluence diagrams for critical pairs to be ``uniform'', in some sense relying on the combinatorics of hypergraph polytopes, see~\cref{contextual-discussion}.
Other contextual families of nestohedra include permutahedra and contextual graph-associahedra, whose term rewriting systems should provide coherence theorems for categorified permutads~\cite{LodayRonco11,Markl19} and reconnectads \cite{DotsenkoKeilthyLyskov}.
Known and unknown structures and coherence theorems are summarized in \cref{table:contextual-hyper}.
Another interesting question would be to characterize combinatorially contextual graph-associahedra and nestohedra, as defined in \cref{ss:examples}. 

\subsection*{Plan of the paper}
In~ \cref{s:hypergraph} we recollect some background on hypergraph polytopes, and we examine their 2-faces in~ \cref{s:anatomy}.
In~\cref{s:rewriting} we introduce some basics on term rewriting, define our hypergraphic rewriting systems and prove our main results establishing a geometric form of Huet's correspondence for nestohedra.
Contextual hypergraphs are introduced and illustrated in  \cref{s:contextual}.

\subsection*{Notations}

We denote by ${\cal R}^*$ the reflexive and transitive closure of a relation ${\cal R}$. 
We use~$|-|$ to denote the cardinality of a set.
We shall manipulate trees of various sorts. They will always be rooted.  We define the  full subtree relation as follows: $\mathfrak{S}$
is a subtree of $\mathfrak{T}$ if $\mathfrak{S}$ is obtained by picking a node of $\mathfrak{T}$ and all its descendants. The subtree relation is traditionally defined by taking connected subsets. Clearly full subtrees are subtrees, but not conversely. We shall only need full subtrees, and as a matter of abbreviation we shall call them just subtrees, following the computer science tradition.

\subsection*{Acknowledgements}
We would like to thank Vincent Pilaud for useful discussions.


\section{Hypergraph polytopes} 
\label{s:hypergraph}

In this section, we recall the definition of hypergraph polytopes. 
We refer to \cite{DP-HP,COI} for more details. 


\subsection{Hypergraphs}
A \defn{hypergraph} is given by a finite set $H$ of \defn{vertices} and a subset of \defn{hyperedges} $\hyper{H}\inc {\cal P}(H)\setminus\{\emptyset\}$ such that $\Union \hyper{H}=H$. 
We say that $\hyper{H}$ is \defn{ordered} if $H$ is equipped with a total order.
We always assume that $\hyper{H}$ is \defn{atomic}, that is $\set{x}\in \hyper{H}$, for all $x\in H$. 
A hyperedge of cardinality 2 is called an \defn{edge}.  
For $X\inc H$, the \defn{plain restriction} of $\hyper{H}$ to $X$ is the set 
$\hyper{H}_X \eqdef   \setc{Z\in \hyper{H}}{\; Z\inc X}$.

We say that $\hyper{H}$ is \defn{connected} if there is no non-trivial partition $H=X_1\union X_2$ such that $\hyper{H}=\hyper{H}_{X_1}\union \hyper{H}_{X_2}$. 
For each hypergraph, there exists a partition $H=X_1\union\ldots\union X_m$ such that each $\hyper{H}_{X_i}$ is connected and $\hyper{H}=\Union(\hyper{H}_{X_i})$.  
The $\hyper{H}_{X_i}$'s are called the \defn{connected components} of $\hyper{H}$.
We say that a non-empty subset $X\inc H$ of vertices is \defn{connected} (resp. a \defn{connected component}) whenever $\hyper{H}_X$ is connected (resp. a connected component of $\hyper{H}$).  Thus our uses of ``connected'' in the sequel will always carry the non-emptyness information.
We denote by $\restrH{H}{X}\eqdef  \hyper{H}_{H\setminus X}$ the plain restriction of $\hyper{H}$ to $H \setminus X$.
The \defn{saturation} of $\hyper{H}$ is the hypergraph
$\Sat(\hyper{H})\eqdef  \setc{X}{\emptyset\incs X\inc H\;\mbox{and}\;\hyper{H}_X\;\mbox{is connected}}$.
A hypergraph is called \defn{saturated} when $\hyper{H}=\Sat(\hyper{H})$.  
The \defn{reconnected restriction} of $\hyper{H}$ to $X$ is the set $$\recrestr{\hyper{H}}{X}\eqdef  \setc{Z\cap X}{Z\in \Sat(\hyper{H}), Z\cap X\neq\emptyset}.$$

\begin{rem}
    Atomic and saturated hypergraphs are called \defn{building sets} in the  literature on nestohedra, see for example \cite{P09,FS05}.
\end{rem}

For $X\inc H$, we will express the fact that $\setc{H_i}{i\in I}$ is the set of connected components of $\restrH{H}{X}$ by the notation $\hyper{H},X  \leadsto  \setc{H_i}{i\in I}$. 
If $I=\{1,\ldots,n\}$, we shall omit the brackets and write $\hyper{H},X  \leadsto  H_1,\ldots,H_n$.
If $\hyper{H}$ is ordered, we morever order the connected components by \defn{increasing order} of their maximal vertices.

If  $x,y,z\in H$ and $\hyper{H},\set{x}\leadsto \setc{H_i}{i\in I}$, we shall write
$$\begin{array}{ll}
\xyz{x}{\hyper{H}}{\set{y,z}} & \mathrm{if}\; y,z\in H_i\; \mbox{for some}\; i \in I\\
\xyz{x}{\hyper{H}}{\set{y},\set{z}} & \mbox{otherwise}.
\end{array}$$
In the second case, we will say that $x$ \defn{disconnects} $y$ and $z$ in $\hyper{H}$. 
The reconnected restriction allows one to characterize the preceding two situations as follows.

\begin{lemma} 
\label{xyz-reconnected} 
We have
$$\begin{array}{lll}
\xyz{x}{\hyper{H}}{\set{y,z}} & \mathrm{iff} & \recrestr{\hyper{H}}{\set{x,y,z}},\set{x}\leadsto\set{y,z} \\
\xyz{x}{\hyper{H}}{\set{y},\set{z}} & \mathrm{iff} & \recrestr{\hyper{H}}{\set{x,y,z}},\set{x}\leadsto \set{y},\set{z}.
\end{array}$$
\end{lemma}

\begin{proof} 
    Let $\hyper{H},\set{x}\leadsto H_1,\ldots,H_n$. 
    Suppose that $\xyz{x}{\hyper{H}}{\set{y,z}}$. 
    Then there exists $i$ such that~$\set{y,z}\inc H_i$, and hence $\set{y,z}\inc H_i\cap\set{x,y,z}$, and in fact $\set{y,z} = H_i\cap\set{x,y,z}$ since~$x\not\in H_i$. 
    Thus $\recrestr{\hyper{H}}{\set{x,y,z}},\set{x}\leadsto\set{y,z}$ holds by definition of reconnected restriction.
    If to the contrary we have $\xyz{x}{\hyper{H}}{\set{y},\set{z}}$, then there exists $i\neq j$ such that $y\in H_i$ and~$z\in H_j$. 
    We then derive likewise that $H_i\cap\set{x,y,z}=\set{y}$ and $H_j\cap\set{x,y,z}=\set{z}$ from which $ \recrestr{\hyper{H}}{\set{x,y,z}},\set{x}\leadsto \set{y},\set{z}$ follows.
    The reverse implications are immediate.
\end{proof}


\subsection{Constructs}
A {connected} hypergraph $\hyper{H}$ gives rise to a set of \defn{constructs}, which are defined inductively as follows.
 
\begin{definition} 
\label{inductive-construct}
Let $\hyper{H}$ be a connected hypergraph and $Y$ be a non-empty subset of $H$.
\begin{enumerate}
\item  If $Y = H$, then the one-node tree $H$ decorated with $H$ is a construct of $\hyper{H}$.
\item If $\hyper{H},Y  \leadsto H_1,\ldots,H_n$, and if $T_1,\ldots,T_n$ are constructs of $\hyper{H}_1,\ldots,\hyper{H}_n$, respectively, then the
non-planar tree $Y(T_1,\ldots,T_n)$ whose root is decorated by $Y$, with $n$ outgoing edges on which the respective $T_i\,$'s are grafted, is a construct.  
\end{enumerate}
In the notation $Y(T_1,\ldots,T_n)$, no order is intended on  $T_1,\ldots,T_n$.
However, when $\hyper{H}$ is ordered, the trees $T_1,\ldots,T_n$ are listed according to the order given by $\hyper{H},Y  \leadsto H_1,\ldots,H_n$, making constructs planar.
When $Y={z}$ is a singleton, we freely write $z$ in place of $\set{z}$.
A \defn{construction} is a construct all of whose nodes are  decorated with singletons. 
\end{definition}

Since all decorations in a construct are disjoint, we freely identify nodes with subsets of~$H$. 
We use the notation $\occ{T}{X}$ to denote the  subtree of $T$ rooted at $X$, defined only if $X$ is indeed a decoration of a node of $T$. 
If $S$ is a subtree of $T$, we denote by $\supp(S)$ the union of the decorations of the nodes of $S$.

\begin{rem} \label{subconstruct-restriction}
The intention behind this presentation is algorithmic: a construct is built by picking and removing a non-empty subset $Y$ of $H$, then branching to the connected components of $\restrH{H}{Y}$ and continuing inductively in all the branches.
It follows readily from the definition that $\occ{T}{X}$ is a construct of $\hyper{H}_{\supp(\occ{T}{X})}$.
\end{rem}

\begin{rem}
    The notion of construct is equivalent to the notion of nested set \cite{P09}, and to the notion of tubing in the case where $\hyper{H}$ is a graph~\cite{CD-CCGA}.  
    We refer to \cite[Sec.~3.1]{COI} for details.
\end{rem}


If $X,Y$ are two nodes of a construct $S$ of $\hyper{H}$, $X$ being the father of $Y$, we can define a new construct $T$ by contracting the edge between $X$ and $Y$, and labeling the resulting vertex of~$T$ by the union of the labels of $X$ and $Y$. 
Formally, if $X$ is a node of $S$ such that $\occ{S}{X}=X(Y(S_1,\ldots,S_m),S_{m+1},\ldots S_n)$, then  $T$ is obtained by replacing in $S$ the  subtree rooted at $X$ with $(X\cup Y)(S_1,\ldots,S_n)$.  
We say that $T$ \defn{covers} $S$ and use the notation $S\prec T$.

\begin{definition} 
  \label{subface-relation}
    We denote $({\cal A}(\hyper{H}),\prec^*)$ the poset of constructs of a connected hypergraph~$\hyper{H}$ obtained as the reflexive and transitive closure of the  covering relation $\prec$.    
\end{definition}
    
We shall need the following lemma.
\begin{lemma} 
  \label{partial-construct}
Let $\hyper{H}$ be a connected hypergraph. 
If $\hyper{H},X\leadsto H_1,\ldots,H_n$ and if constructs $T_i$ of~$\hyper{H}_i$ are given for all $i\in I\subseteq \set{1,\ldots,n}$, then $X(\setc{T_i}{i\in I})$ is a construct of $\recrestr{\hyper{H}}{(H\setminus\cup\setc{H_j}{j\not\in I})}$.
\end{lemma}
\begin{proof} This is an immediate consequence of the following two observations: we have
$$\recrestr{\hyper{H}}{(H\setminus\cup\setc{H_j}{j\not\in I})},X\leadsto \setc{H_i}{i\in I}\;\; \mbox{and}\;\; (\recrestr{\hyper{H}}{(H\setminus\cup\setc{H_j}{j\not\in I})})_{H_i}=\hyper{H}_{H_i}=\hyper{H}_i.$$
\end{proof}


\subsection{Hypergraph polytopes}
\label{ss:hypergraph-polytopes}

We are now ready to define hypergraph polytopes, a.k.a nestohedra.

\begin{definition}
    A \defn{hypergraph polytope} is a polytope whose face lattice is isomorphic to the poset of constructs of some connected hypergraph $\hyper{H}$.
\end{definition}

Do\v sen and Petri\'c gave polytopal realizations of hypergraph polytopes in ~\cite{DP-HP}.
The idea is that the connected subsets of $\hyper{H}$ specify the faces of a fixed $(|H|-1)$-dimensional simplex that are to be truncated.

Let us introduce a few families of hypergraph polytopes that will be studied in \cref{ss:examples}.
When $\hyper{H}$ is a graph, hypergraph polytopes are known as \emph{graph-associahedra} \cite{CD-CCGA}.

\subsubsection{Simplices}
The $n$-dimensional \defn{simplex} is the realization of the hypergraph 
$$\mathbb{S}_n\eqdef  \{\{1\},\{2\},\ldots,\{n+1\},\{1,\ldots,n+1\}\}.$$

\subsubsection{Cubes}
The $n$-dimensional \defn{cube} is the realization of the hypergraph
$$\mathbb{C}_n\eqdef  
\{\{1\},\{2\},\ldots,\{n+1\}\}\cup\{\{j \ | \ 1 \leq j \leq i \} \ | \ 1 \leq i \leq n+1\}.$$
Note that these two families of hypergraph polytopes are defined by genuine hypergraphs, contrary to the following families of graph-associahedra.

\subsubsection{Associahedra}
The $n$-dimensional \defn{associahedron} is the realization of the graph 
$$\mathbb{K}_n\eqdef  \{\{1\},\{2\},\ldots,\{n+1\},\{1,2\},\ldots,\{n,n+1\}\}.$$
In other words, it is the graph-associahedron for the linear graph with $n+1$ vertices.
By adding the edge $\{n+1,1\}$, we get the $n$-dimensional \defn{cyclohedron}, which plays a role in our story (see~\cref{non-contextual-2}).  

\subsubsection{Permutahedra}
The $n$-dimensional \defn{permutahedron} is the realization of the graph 
$$\mathbb{P}_n\eqdef  \{\{1\},\{2\},\ldots,\{n+1\}\} \cup \{\{i,j\} \ | \ 1 \leq i \neq j \leq n+1 \}.$$
In other words, it is the graph-associahedron for the complete graph on $n+1$ vertices.

\smallskip
Constructs of associahedra and permutahedra are easily seen to be in bijection with planar rooted trees and surjections, respectively.
These standard labellings can be found for instance in \cite{PalaciosRonco}. 
In brief:
\begin{itemize} 
\item given a planar rooted tree with $n+2$ leaves $l_0,\ldots,l_{n+1}$, the corresponding construct of $\mathbb{K}_n$ is obtained by placing $i$ in the first common node between the two paths from~$l_{i-1}$ and~$l_i$ to the root, for each $1 \leq i \leq n+1$, and then by removing the leaves;
\item given a construct $X_r(X_{r-1}(\ldots(X_0)\ldots))$ of $\mathbb{P}_n$, one defines a surjection $s:\set{0,\ldots,n}\rightarrow \set{0,\ldots,r}$ by the formula $s(i)=j$ whenever $i+1\in X_j$.
\end{itemize}
More details can be found in \cite[Section 2.6]{COI}.

\smallskip
Note that the definition of $\hyper{S}_n$ and $\hyper{P}_n$ does not depend on any order on the vertices, while the definition of $\hyper{C}_n$ and $\hyper{K}_n$ involves the total order on~$\set{1\ldots,n+1}$. 

\begin{rem}
  In the above descriptions, we can replace $\set{1,\ldots,n+1}$ by any finite set $X=\set{x_1,\ldots,x_{n+1}}$  (resp. any finite linear order $X=x_1<\cdots<x_{n+1}$)  and define $\hyper{S}^X$ and $\hyper{P}^X$ (resp. $\hyper{C}^X$, $\hyper{K}^X$) accordingly.
\end{rem}

\subsubsection{Operahedra}
To every  planar tree $\calT$ with $n+2$ vertices, one can associate its $n$-dimensional \defn{operahedron}, whose faces are in bijection with the nestings of $\calT$ \cite{laplante-anfossiDiagonalOperahedra2022a,CLA1}, which in turn
are in bijection with the tubings of the \defn{line graph} $\mathbb{L}(\calT)$ of $\calT$.
The vertices of $\mathbb{L}(\calT)$ are the edges of ${\cal T}\!$, and two vertices are connected whenever as edges of ${\cal T}$ they share a common vertex, see \cref{fig:line-graph}.
In other words, an $n$-dimensional operahedron is the graph-associahedron of a clawfree block graph with $n+1$ vertices.

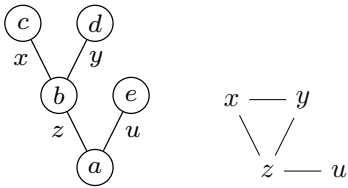
\begin{figure}[h!]
  \begin{center}
    \begin{tabular}{ccc}
    \resizebox{2cm}{!}{\begin{tikzpicture}[scale=0.8]
        \node (E)[circle,draw=black,minimum size=4mm,inner sep=0.1mm] at (0,0) {\scriptsize $a$};
        \node (F) [circle,draw=black,minimum size=4mm,inner sep=0.1mm] at (-0.5,1) { \scriptsize $b$};
        \node (A) [circle,draw=black,minimum size=4mm,inner sep=0.1mm] at (0.5,1) {\scriptsize $e$};
        \node (Asubt) [circle,draw=black,minimum size=4mm,inner sep=0.1mm] at (-1,2) {\scriptsize  $c$};
        \node (P) [circle,draw=black,minimum size=4mm,inner sep=0.1mm] at (0,2) {\scriptsize $d$};
        \draw[-] (E)--(F) node  [midway,left] {\scriptsize $z$};
        \draw[-] (E)--(A) node  [midway,right] {\scriptsize $u$};
     \draw[-] (F)--(Asubt) node [midway,left] {\scriptsize $x$};
     \draw[-] (F)--(P) node [midway,right] {\scriptsize $y$};
       \end{tikzpicture}}
    
    &&
    \resizebox{2cm}{!}{
    \begin{tikzpicture}
        \node (Z)[] at (-0.5,0) {$z$};
        \node (U)[]  at (0.5,0) {$u$};
        \node (X)[]  at (-1,1) {$x$};
        \node (Y)[]  at (0,1) {$y$};
        \draw[-] (Z)--(U) node  {};
     \draw[-] (Z)--(X) node  {};
     \draw[-] (Z)--(Y) node {};
     \draw[-] (X)--(Y) node {};
       \end{tikzpicture}}
    \end{tabular}
    \end{center}
    \caption{A planar tree with $5$ vertices (left) and its line graph (right).}
    \label{fig:line-graph}
\end{figure}

Many more examples are to be found in~\cite{DP-HP,COI,CDOO}, as well as in the abundant literature on nestohedra.


\section{Anatomy of the 2-skeleton} 
\label{s:anatomy}

In this section we describe all the possible $2$-faces of a hypergraph polytope.
These will be associated with the local confluence diagrams of a rewriting system on constructions in \cref{ss:critical}.


\subsection{Two types of two-faces} 
\label{two-types}
The \defn{dimension} of a construct $T$, or equivalently of its corresponding face in the associated hypergraph polytope, is given by
$$\dim T \eqdef \sum_{X\:\mathrm{node\: of}\: T}(|X|-1).$$
In particular, constructions have dimension $0$. 
Constructs of dimension $1$ have a single non-singleton node of the form $\set{x,y}$. 
Constructs $T$ of dimension $2$ are of two kinds:
\begin{itemize}
\item[(A)]  $T$ has exactly two non-singleton nodes $\set{x_1,x_2}$ and $\set{y_1,y_2}$, both of cardinality $2$. 
\item[(B)] $T$ has exactly one non-singleton node $\set{x_1,x_2,x_3}$ of cardinality $3$.
\end{itemize}

\subsubsection*{Type (A)}
\label{ss:typeA}
If $T$ is of type (A), we get the following generic picture.
$$ 
\xymatrix @-1.65pc { x_1(x_2)\cdots y_1(y_2) \ar @{-}[dddd]_{x_1(x_2)\cdots \set{y_1,y_2}} 
 \ar @{-}[rrrr]_{\set{x_1,x_2}\cdots y_1(y_2)} &&&& x_2(x_1)\cdots y_1(y_2) \ar @{-}[dddd]^{x_2(x_1)\cdots \set{y_1,y_2}} \\
 &&&&\\
 && \set{x_1,x_2} \cdots \set{y_1,y_2}&&\\
 &&&&\\
 x_1(x_2)\cdots y_2(y_1)  \ar @{-}[rrrr]_{\set{x_1,x_2}\cdots y_2(y_1)} &&  && x_2(x_1)\cdots y_2(y_1)}
 $$
The construct $T$ which corresponds to the 2-face above has two distinct non-singleton nodes $\set{x_1,x_2}$ and $\set{y_1,y_2}$, whence the schematized expression  $\set{x_1,x_2} \cdots \set{y_1,y_2}$.
All the other constructs are obtained by replacing in $T$ one or two of these nodes, say $\set{x_1,x_2}$, by a tree~$x_1(x_2)$ or $x_2(x_1)$ and redistributing the children of $\set{x_1,x_2}$ as children of $x_1$ or $x_2$, in a unique way dictated by connectivity. 

\subsubsection*{Type (B)}
\label{ss:typeB}
If $T$ is of type (B), then, up to permutation of $x_1,x_2,x_3$, we get four possible shapes corresponding to the number $N$ of  elements in $X\eqdef \set{x_1,x_2,x_3}$ that disconnect the other two in~$\hyper{K}\eqdef \hyper{H}_{\supp(\occ{T}{X})}$.
Like in the previous case, the constructs of the edges and vertices of the $2$-face corresponding to~$T$ result from replacing in $T$ the node $X$ with the indicated respective trees and redistributing uniquely the  children of $X$ (see also~\cref{instance-construct}).
 
\smallskip\noindent
(B1)  When $N=3$, that is when $\xyz{x_1}{\hyper{K}}{\set{x_2}\!,\! \set{x_3}}$, $\xyz{x_2}{\hyper{K}}{\set{x_1}\!,\! \set{x_3}}$, and
 $\xyz{x_3}{\hyper{K}}{\set{x_1}\!,\! \set{x_2}}$, we have
 
 $$\xymatrix @-1.65pc {&& x_1(x_2,x_3) \ar @{-}[ddll]_{\set{x_1,x_2}(x_3)} \ar @{-}[ddrr]^{\set{x_1,x_3}(x_2)}&& \\
 &&&&\\
 x_2(x_3,x_1) \ar @{-}[rrrr]^{\set{x_2,x_3}(x_1)} &&&& x_3(x_1,x_2)}$$

\smallskip\noindent
(B2) When $N=2$, that is when $\xyz{x_1}{\hyper{K}}{\set{x_2}\!,\! \set{x_3}}$, $\xyz{x_2}{\hyper{K}}{\set{x_1,x_3}}$ and
 $\xyz{x_3}{\hyper{K}}{\set{x_1}\!,\! \set{x_2}}$, we have
 
 $$ 
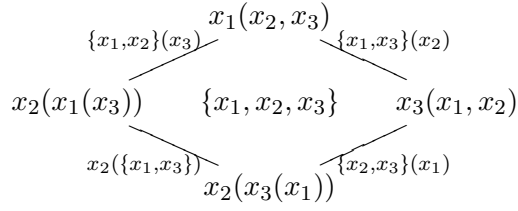
\xymatrix @-1.65pc {&& x_1(x_2,x_3) \ar @{-}[ddll]_{\set{x_1,x_2}(x_3)} \ar @{-}[ddrr]^{\set{x_1,x_3}(x_2)}&& \\
 &&&&\\
 x_2(x_1(x_3)) \ar @{-}[ddrr]_{x_2(\set{x_1,x_3})}&& \set{x_1,x_2,x_3}  && x_3(x_1,x_2)\ar @{-}[ddll]^{\set{x_2,x_3}(x_1)}\\
 &&&&\\
 &&  x_2(x_3(x_1))&&}$$

\smallskip\noindent
(B3) When $N=1$, that is when $\xyz{x_1}{\hyper{K}}{\set{x_2}\!,\! \set{x_3}}$, $\xyz{x_2}{\hyper{K}}{\set{x_1,x_3}}$ and
 $\xyz{x_3}{\hyper{K}}{\set{x_1,x_2}}$, we have
 
$$
 
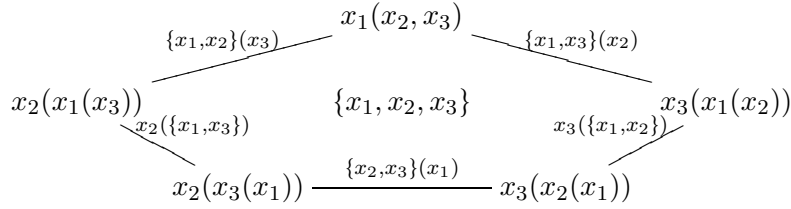
\xymatrix @-1.65pc {&& x_1(x_2,x_3) \ar @{-}[ddll]_{\set{x_1,x_2}(x_3)} \ar @{-}[ddrr]^{\set{x_1,x_3}(x_2)}&& \\
 &&&&\\
 x_2(x_1(x_3)) \ar @{-}[ddr]^{x_2(\set{x_1,x_3})}&& \set{x_1,x_2,x_3}  && x_3(x_1(x_2))\ar @{-}[ddl]_{x_3(\set{x_1,x_2})}\\
 &&&&\\
 &x_2(x_3(x_1))\ar @{-}[rr]^{\set{x_2,x_3}(x_1)}&  & x_3(x_2(x_1))&}
$$

\smallskip\noindent
(B4) When $N=0$, that is when $\xyz{x_1}{\hyper{K}}{\set{x_2,x_3}}$, $\xyz{x_2}{\hyper{K}}{\set{x_1,x_3}}$ and
 $\xyz{x_3}{\hyper{K}}{\set{x_1,x_2}}$, we have
 
 $$\xymatrix @-1.65pc {&& x_1(x_2(x_3)) \ar @{-}[ddll]_{\set{x_1,x_2}(x_3)} \ar @{-}[ddrr]^{x_1(\set{x_2,x_3})}&& \\
 &&&&\\
 x_2(x_1(x_3)) \ar @{-}[dd]^{x_2(\set{x_1,x_3})}&& \set{x_1,x_2,x_3} && x_1(x_3(x_2))\ar @{-}[dd]_{\set{x_1,x_3}(x_2)}\\
 &&  &&\\
 x_2(x_3(x_1)) \ar @{-}[ddrr]_{\set{x_2,x_3}(x_1)} &&  && \ar @{-}[ddll]^{x_3(\set{x_1,x_2})} x_3(x_1(x_2))\\
 &&&&\\
 &&  x_3(x_2(x_1))&&}
$$

By~\cref{xyz-reconnected}, we can read these pictures as describing the 
respective reconnected restrictions $\recrestr{\hyper{K}}{X}$ of $\hyper{K}$:
$$\begin{array}{lll}
\mathrm{(B1)} & \set{\set{x_1},\set{x_2},\set{x_3},\set{x_1,x_2,x_3}} & \mbox{(2-simplex)}\\
\mathrm{(B2)} & \set{\set{x_1},\set{x_2},\set{x_3},\set{x_1,x_3},\set{x_1,x_2,x_3}} & \mbox{(2-cube})\\
\mathrm{(B3)} & \set{\set{x_1},\set{x_2},\set{x_3},\set{x_1,x_3},\set{x_1,x_2},\set{x_1,x_2,x_3}} & \mbox{(2-associahedron)}\\
\mathrm{(B4)} & \set{\set{x_1},\set{x_2},\set{x_3},\set{x_1,x_3},\set{x_1,x_2},\set{x_2,x_3},\set{x_1,x_2,x_3}} & \mbox{(2-permutahedron)}
\end{array}$$

Incidentally, these  hypergraphs witness the fact that there do exist 2-faces of type (B) of  each  of these types: take $\hyper{H}$ to be one of those four $2$-dimensional hypergraphs, and $T$ to be their unique construct of dimension 2.


\subsection{$X$-faces and shapes}

In the rest of this section, we make the discussion above on $\cdots$ more formal.
For a 3-element subset $X=\set{x_1,x_2,x_3}$ of $H$, we say that a 2-face viewed as a contruct $T$ of $\hyper{H}$ is an \defn{$X$-face} if its unique non-singleton node is decorated by $X$.  By letting $X$ range over all subsets of $H$ of cardinality 3, we form in this way a partition of all 2-faces of type (B).
If $X$ is the root of $T$, the following  lemma invites us to see $T=X(\ldots)$ as an ``instantiation'' of~$X$ (viewed as the maximum face of $\recrestr{\hyper{H}}{X}$).

\begin{lemma} 
  \label{instance-construct} 
  If $\hyper{H}$ is a connected hypergraph, if $X$ is a subset of $H$ such that $|X|=3$ and~$T$ is a 2-dimensional construct with root $X$, then the map $\psi$ from the poset of subfaces of $T$ to the poset of faces of $\recrestr{\hyper{H}}{X}$, defined on a face $S$ by pruning all nodes of $S$ that are not subsets of $X$, is an order-isomorphism.
\end{lemma}

\begin{proof}
We shall build an inverse $\phi$ of $\psi$.
When discussing type (B) 2-faces, we have described up to permutation all the possible connected hypergraphs on the set $X$ of vertices and their respective posets of faces. 
We will treat the case where ${\cal A}(\recrestr{\hyper{H}}{X})$ is the poset described in case (B2). The other cases are similar.
Pick the 0-face~$S\eqdef  x_1(x_2,x_3)$. 
We map $S$ to a $0$-face $\phi(S)$ of $T$ as follows. 
Let $\hyper{H},x\leadsto \{H_1,\ldots,H_n\}$. 
By \cref{xyz-reconnected}, we have $\xyz{x_1}{\hyper{H}}{\set{x_2},\set{x_3}}$, hence we have, say $x_2\in H_1$ and $x_3\in H_2$. Then let
  $\hyper{H_1},x_2\leadsto \{H_{1,1},\ldots,H_{1,p}\}$ and $\hyper{H_2},x_3\leadsto \{H_{2,1},\ldots,H_{2,q}\}$.
  Then we have $\hyper{H},X\leadsto \{H_{1,1},\ldots,H_{1,p},H_{2,1},\ldots,H_{2,q},H_3,\ldots H_n\}$, so that $T$ writes as
  $$T= X(T_{1,1},\ldots,T_{1,p},T_{2,1},\ldots,T_{2,q},T_3,\ldots T_n).$$ 
All these data determine uniquely a 0-dimensional subface of $T$, namely
  $$\phi(S)\eqdef x_1(x_2(T_{1,1},\ldots,T_{1,p}),x_3(T_{2,1},\ldots,T_{2,q}),T_3,\ldots T_n).$$
It is plain that $\psi(\phi(S))=S$.
  The same applies to all other 0-dimensional (resp. 1-dimensional) faces of $\recrestr{\hyper{H}}{X}$, establishing $\phi$ as a bijection, which is also easily seen to be monotonic: for example, we have
 $$\phi(\set{x_1,x_2}(x_3))= \set{x_1,x_2}(x_3(T_{2,1},\ldots,T_{2,q}),T_{1,1},\ldots,T_{1,p},T_3,\ldots T_n),$$
 evidencing $\phi(x_1(x_2,x_3))\preceq \phi(\set{x_1,x_2}(x_3))$. 
 The map $\psi$ is also monotonic, since the above pruning does not affect the place where the contraction occurs -- e.g., the edge of $\phi(x_1(x_2,x_3))$ that is contracted to get $\phi(\set{x_1,x_2}(x_3))$ is the edge between $x_1$ and $x_2$, which can thus be contracted in the preimage $x_1(x_2,x_3)$ to yield the preimage $\set{x_1,x_2}(x_3)$.
 Finally, we set
 $$\phi(\set{x_1,x_2,x_3})\eqdef \set{x_1,x_2,x_3}(T_{1,1},\ldots,T_{1,p},T_{2,1},\ldots,T_{2,q},T_3,\ldots T_n)=T,$$
 which concludes the proof.
 \end{proof}

 \begin{corollary} 
  \label{instance-construct-general} 
  If $\hyper{H}$ is a connected hypergraph, if $X$ is a subset of $H$ such that $|X|=3$ and $T$ is an $X$-face of $\hyper{H}$, then the poset of faces of $T$ is isomorphic to the poset of faces of~$\recrestr{(\hyper{H}_{\supp(\occ{T}{X})})}{X}$.
\end{corollary}

\begin{proof} This is an immediate consequence of the previous lemma and of the observation that the poset of faces of $T$ is isomorphic to the poset of faces of $\occ{T}{X}$.
\end{proof}
 
Based on this corollary, we shall say, for any $X$-face $T$, that $T$ has \defn{shape} 
$\recrestr{(\hyper{H}_{\supp(\occ{T}{X})})}{X}$.


\section{Hypergraphic rewriting systems} 
\label{s:rewriting}

We associate to each hypergraph two term rewriting systems, one given on its constructs, and the other given on its constructions.
We show that the latter is terminating and confluent, and describe its critical pairs. 


\subsection{Recollections on rewriting} \label{recollection-section}

A \defn{signature} $\Sigma$ is a tuple $(V,F,S,\ari,\outsort,\insort)$ made of 
\begin{itemize}
  \item a set $V$ of \defn{variables},
  \item a non-empty set $F$ of \defn{function symbols}, and
  \item a set $S$ of \defn{sorts},
\end{itemize}
together with an \defn{arity}, \defn{output sort} and \defn{input sort} functions
\begin{itemize}
  \item $\ari : F \to \mathbb{N}$,
  \item $\outsort : F \cup V \to S$,
  \item $\insort : F \to \Sigma_{n \geq 0} S^{n}$, such that for $f \in F$, we have $\insort(f) \in S^{\ari(f)}$. 
\end{itemize}
The $i$th component of $\insort(f)$ is denoted $\insort(f,i)$.
The set $\Ter(\Sigma)$ of \defn{terms} over a signature $\Sigma$ is defined inductively as follows. 
\begin{enumerate}
  \item If $t \in V$ is a variable, then $t$ is a term.
  \item If $f \in F$ is an arity $n$ function symbol, and $t_1,\ldots,t_n$ are terms such that $\outsort(t_i)=\insort(f,i)$, then $f(t_1,\ldots,t_n)$ is a term, and $\outsort(f(t_1,\ldots,t_n))\eqdef \outsort(f)$.
\end{enumerate}
We say that a term \defn{has sort}  $\outsort(t)$.
For a term $t \in \Ter(\Sigma)$, its set of \defn{variables} is defined as 
\begin{equation*}
  \var(t) \eqdef  
  \begin{cases}
    \{t\} & \text{ if } t \in V, \\
    \bigcup_{1 \leq i \leq n}\var(t_i) & \text{ if } t=f(t_1,\ldots,t_n).
  \end{cases}
\end{equation*}
A term $t$ is \defn{linear} if all its variables occur only once in it, i.e.\ if the unions in the inductive definition of $ \var(t)$ are disjoint.
A term $t$ is \defn{closed} (resp.  \defn{open}) if it does not contain any variable, i.e., if $\var(t)=\emptyset$ (resp. $\var(t)\neq\emptyset$).
A \defn{rewriting  rule} over $\Sigma$ is an ordered pair $(l,r)$ of terms in $\Ter(\Sigma)$, denoted $l \to r$, such that
$\outsort(l)=\outsort(r)$ and  the following conditions hold:
\begin{enumerate}
  \item the first term $l$ is not a variable, that is $l \notin V$,
  \item the variables of the second term are already in the first term, that is $\var(r) \subseteq \var(l)$.
\end{enumerate}

\begin{definition}
  A many-sorted \defn{term rewriting system} is an ordered pair $(\Sigma,R)$ made of a signature and of a set of rewriting rules $R$ over $\Sigma$.
\end{definition}

A \defn{context} $C[-]$ is defined inductively as follows.
\begin{enumerate}
  \item We give ourselves symbols $[-]_a$ for~$a\in S$, which we declare to be contexts with $\outsort([-]_a) \eqdef a$.
  \item If $f \in F$ is an arity $n$ function symbol, $t_1,\ldots,t_{i-1},t_{i+1},\ldots,t_n$ are terms such that $\outsort(t_j)=\insort(f,j)$ and $C$ is a context such that $\outsort(C)=\insort(f,i)$, then 
  $$f(t_1,\ldots,t_{i-1},C,t_{i+1},\ldots,t_n)$$
   is a context, and $\outsort(f(t_1,\ldots,t_n))\eqdef \outsort(f)$.
\end{enumerate}
Note that by construction a context contains exactly one symbol $[-]_a$ for some $a \in S$, called the \defn{hole}.
Given such a context $C$ and a term $t$ with $\outsort(t)=a$, then we denote by~$C[t]$ the term obtained by  filling the hole with $t$, i.e., by replacing $[-]_a$ by $t$ in $C$.

A \defn{substitution} is a map $\sigma : \Ter(\Sigma) \to \Ter(\Sigma)$ which satisfies $$\sigma(f(t_1,\ldots,t_n))=f(\sigma(t_1),\ldots,\sigma(t_n))$$ for all terms $f(t_1,\ldots,t_n)$ in $\Ter(\Sigma)$.
That is, a substitution is completely determined by its value on variables. We say that $t'$ is an \defn{instance} of $t$ if
$t'=\sigma(t)$ for some $\sigma$, and that $t'$ is a \defn{closed instance} of $t$ if moreover $t'$ is closed.

A rewriting rule $l \to r$ determines a set of \defn{rewritings} $\sigma(l)\to \sigma(r)$ for all substitutions $\sigma$. 
We say that $\sigma(l)\to \sigma(r)$ is an \defn{instance} of $l \to r$.
These in turn give rise to \defn{rewriting steps} $C[\sigma(l)] \to C[\sigma(r)]$, whenever $C[\sigma(l)]$ is defined. 
We say that $C[\sigma(l)] \to C[\sigma(r)]$ is an \defn{instance in context} of the rewriting rule $l \to r$.

A rewriting system is \defn{terminating} if it admits  no infinite  reduction sequence, that is, no sequence $t=t_0\to t_1\to\cdots\to t_n \to\cdots$ such that, for each $i$, $t_i\to t_{i+1}$ is a rewriting step.
A term $l \in \Ter(\Sigma)$ is \defn{reducible} if there exists an $r \in \Ter(\Sigma)$ such that $l \to r$; otherwise it is called  \defn{irreducible}.
We say that $r$ is a \defn{normal form} of $l$ if $l \to^* r$ and $r$ is irreducible.

We say that $(\Sigma,R)$ is \defn{locally confluent} (resp.\ \defn{confluent}) if for all $s,t_1,t_2 \in \Ter(\Sigma)$ such that $t_1 \leftarrow s \to t_2$ (resp.\ $t_1 \:{}^*\!\!\leftarrow s \to^* t_2$), there exists a term $t$ with $t_1 \to^* t\; {}^*\!\!\leftarrow t_2$.  
The diagram  
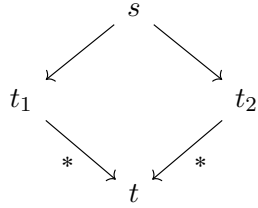
\begin{center}
  \begin{tikzcd}
    & s \arrow[ld] \arrow[rd] &                     \\
t_1 \arrow[rd, "*"'] &                         & t_2 \arrow[ld, "*"] \\
    & t                       &                    
\end{tikzcd}
\end{center}
is called a \defn{local confluence diagram}.

There are two cornerstone lemmas in term rewriting theory (\cref{thm:Newman,l:local-confluence-critical-pair} below).

\begin{lemma}[Newman's lemma]
  \label{thm:Newman}
  If $(\Sigma,R)$ is terminating, then it is confluent if and only if it is locally confluent.
\end{lemma}

We can further characterize local confluence. 
A pair of reduction steps $s \to t_1$ and $s\to t_2$ is said to form a \defn{critical pair}   if 
\begin{enumerate}
  \item there are no terms $s',t_1',t_2'$ and context $C[-]$ such that $s' \to t_1'$, $s' \to t_2'$, $s=C[s']$, $t_1=C[t_1']$ and $t_2=C[t_2']$.
  \item there are no terms $s',t_1',t_2'$ and substitution $\sigma$ such that $s' \to t_1'$, $s' \to t_2'$, $s=\sigma(s')$, $t_1=\sigma(t_1')$ and $t_2=\sigma(t_2')$.
\end{enumerate}
A local confluence diagram  for a critical pair is called a \defn{critical confluence diagram}.
It can be shown that every pair of rewriting steps  $s \to t_1$ and $s\to t_2$ falls in exactly one of the following situations (up to permuting $t_1$ and $t_2$):
\begin{enumerate}
\item[(a1)] One can write $s=D[s_1]^1[s_2]^2$ and there exist $t'_1$, $t'_2$ such that $s_1\to t'_1$, $s_2\to t'_2$, $t_1=D[t'_1]^1[s_2]^2$ and $t_2=D[s_1]^1[t'_2]^2$. 

\smallskip\noindent
Here $D$ is a context with two holes. The inductive definition  builds on the above defined terms and  contexts (rebaptised  as contexts with one hole), as follows: 
\begin{enumerate}
 \item[(1)] If $f \in F$ is an arity $n$ function symbol, with $n>1$, $t_1,\ldots,t_{i-1},t_{i+1},$ $\ldots,$ $t_{j-1},t_{j+1},\ldots,t_n$ are terms such that 
 $\outsort(t_j)=\insort(f,j)$, $C^1$ and $C^2$  are contexts with one hole such that $\outsort(C^1)=\insort(f,i)$ and $\outsort(C^2)=\insort(f,j)$
 then 
  $$f(t_1,\ldots,t_{i-1},C^1,t_{i+1},\ldots,t_{j-1},C^2,t_{j+1},\ldots,t_n)$$
   is a context with two holes, and $\outsort(f(t_1,\ldots,t_n))\eqdef \outsort(f)$.  We label the hole of $C^1$ (resp. $C^2$) with un upperscript $1$ (resp. $2$).
 \item[(2)] If $f \in F$ is an arity $n$ function symbol, $t_1,\ldots,t_{i-1},t_{i+1},\ldots,t_n$ are terms such that $\outsort(t_j)=\insort(f,j)$ and $C$ is a context with two holes such that $\outsort(C)=\insort(f,i)$, then 
  $$f(t_1,\ldots,t_{i-1},C,t_{i+1},\ldots,t_n)$$
 is a context with two holes, and $\outsort(f(t_1,\ldots,t_n))\eqdef \outsort(f)$.
\end{enumerate}
We define the filling of a context with two holes in the obvious way.
\item[(a2)] $s,t_1,t_2$ are such that there exists a substitution $\sigma$, a variable $x$, a context $C$  and two terms $t'_1$ and $t'_2$ such that
$s=C[\sigma(s')]$, $s'\to t'_1$, $t_1=C[\sigma(t'_1)]$, $\sigma(x)\to t'_2$ and 
$t_2=C[\sigma'(s')]$, where $\sigma'(x)\eqdef t'_2$ and $\sigma'(y)\eqdef \sigma(y)$ for $y\not = x$.
\item[(b)] $s,t_1,t_2$ are such that there exists a critical pair $s'\rightarrow t'_1$, $s'\rightarrow t'_2$, a substitution $\sigma$ and a context $C$ such that $s=C[\sigma(s')]$, $t_1=C[\sigma(t'_1)]$, and $t_2=C[\sigma(t'_2)]$.
\end{enumerate}
We refer to \cref{Huet-MacLane} for an illustration of these three cases in the seminal rewriting system underlying Mac Lane's coherence theorem.
Situation (b) above is called an \defn{overlapping}, and critical pairs are alternatively called minimal overlappings.
It can also be shown that all pairs of type~(a1) or~(a2) above can always be completed into local confluence diagrams, that is, they \defn{converge}. 
For pairs of type (a1), those  diagrams are squares, i.e. $t_1 \to t \leftarrow t_2$. 
If the system is \defn{linear}, that is if all variables on each side of the rewriting rules are distinct, then the local confluence diagrams of type~(a2) are also squares.
Finally, we observe that if critical confluence diagrams exist for all critical pairs, then, by taking their instantiations in context, we get all local confluence diagrams of type (b).  These are the ingredients of the proof of the following second cornerstone lemma.

\begin{lemma}[Knuth-Bendix lemma]
  \label{l:local-confluence-critical-pair}
  A term rewriting system is locally confluent if and only if every critical pair is convergent. 
\end{lemma}

Thus, for a terminating rewriting system, it suffices to check that all critical pairs converge to conclude that the system is confluent. 

\smallskip
We shall need the following lemma.

\begin{lemma} \label{closed-termination}
If the signature $\Sigma$ is such that for every sort $s$ there exists a closed term $t$ such that $\outsort(t)=s$, and if  every closed term $t$ is such that every reduction sequence starting from $t$ terminates, then $(\Sigma,R)$ is terminating.
\end{lemma}
\begin{proof} 
Let $t$ be an arbitrary term. By our assumption, there exists a substitution $\sigma$ such that $\sigma(t)$ is closed. 
The result then follows immediately by contradiction: any infinite reduction sequence from 
$t$ would reflect in an infinite reduction sequence for $\sigma(t)$.
\end{proof}


\subsection{Rewriting on constructs}
\label{ss:rewriting-constructs}
We define our first term rewriting system on constructs. 
From now on, we will consider only \emph{ordered} hypergraphs.

\begin{definition} \label{def:signature-hyper}
  Let $\hyper{H}$ be an ordered connected hypergraph. 
Consider the \defn{signature} $\Sigma_\hyper{H}$ made of the following data: 
\begin{itemize}
  \item The variables are the connected subsets of $H$, that is $V\eqdef \{ X \subseteq H \ | \ \hyper{H}_{X} \text{ is connected}\}$. 
  \item The function symbols are ordered pairs $(X,Y)$ where Y is a connected subset of $H$ and  X is a non-empty subset of $Y$:
  $$F\eqdef \{(X,Y) \ | \ \emptyset \neq X \subseteq Y \subseteq H, \ \hyper{H}_Y \text{ is connected}\}.$$
  \item The sorts are the connected subsets of $H$, i.e., $S\eqdef \{ X \subseteq H \ | \ \hyper{H}_X \text{ is connected}\}$.
  \item For $(X,Y) \in F$, we define $\ari(X,Y)$ as the number of connected components of~$\hyper{H}_Y\setminus X$.
  \item Variables $X \in V$ are their own sort, i.e., $\outsort(X)=X$, while function symbols $(X,Y) \in F$ are of sort $\outsort(X,Y)=Y$.
  \item For function symbols $(X,Y) \in F$ such that $\hyper{H}_Y,X \leadsto Y_1,\ldots,Y_n$, and for $1 \leq i \leq n$, we set $\insort((X,Y),i)=Y_i$.
\end{itemize}
\end{definition}

\begin{rem} \label{open-to-closed}
  Note that, according to this definition, the variables appearing in a term are always distinct. 
  Therefore, we can unambiguously identify them (as we did, and will continue to do) with their  sort.
We 
observe that any term $t$ may be considered as a closed term $\overline{t}$, by replacing all its variables $Y$ with $0$-ary symbols $(Y,Y)$. Formally, $\overline{t}=\sigma(t)$ where, for all variables $Y$ of $t$, $\sigma(Y)=(Y,Y)$.
\end{rem}

\begin{lemma} 
  \label{l:bijection-terms}
  There is a bijection between the set of closed terms of sort $H$ over $\Sigma_\hyper{H}$ and the set of constructs of $\hyper{H}$, defined by projecting all function symbols $(X,Y)$ to their first component $X$.
\end{lemma}

\begin{proof}
  We compare the inductive definition of constructs (\cref{inductive-construct}) with the inductive definition of $\Ter(\Sigma_\hyper{H})$ above.
  First observe that there is only one arity $0$ function symbol of output sort $H$, namely  $(H,H)$. 
  We associate to this term the non-planar tree with only one node, decorated by $H$.
  Now, an arity $n$ function symbol of output sort $H$ is an ordered pair $(Y,H)$ with $Y \subseteq H$, with
  $\hyper{H},Y \leadsto H_1,\ldots,H_n$, and a valid closed term $(Y,H)(t_1,\ldots,t_n)$ is obtained from closed terms $t_i$ of  sort $H_i$, for $1 \leq i \leq n$. 
  We associate to the term $(Y,H)(t_1,\ldots,t_n)$ the non-planar tree $Y(T_1,\ldots,T_n)$ with root decorated by $Y$, and the non-planar trees~$T_i$, associated by induction to the various $t_i$'s, grafted on its leaves.
  It is clear from the inductive nature of the definitions that this correspondence is bijective.
\end{proof}

The correspondence between closed terms and constructs extends to all terms.

\begin{lemma} 
  \label{l:bijection-open}
  For every non-empty subset $X\inc H$,
 there is a bijection $\chi$ between the set of  terms $t$ of  sort $H$ over $\Sigma_\hyper{H}$  such that $\bigcup \var(t)=H\setminus X$ and the set of constructs of $\recrestr{\hyper{H}}{X}$, 
 defined by projecting all function symbols $(Z,Y)$ to their first component $Z$ and by pruning all variables.
\end{lemma}

\begin{proof}
The proof is a variation on the proof of \cref{l:bijection-terms}, and subsumes it (take $X=~H$).  
Let $t$ be a term  of  sort $H$. 
It cannot be a variable, as it would have to be $H$, contradicting the non-emptyness of $X$. 
Thus, ignoring the order of subterms, $t$ has the form $(Y,H)(\setc{t_i}{i\in I}\cup\setc{H_j}{j\not\in I})$, where $\hyper{H},Y \leadsto H_1,\ldots,H_n$ and $I\inc\set{1,\ldots,n}$, and we conclude by induction together with \cref{partial-construct}.
\end{proof}

We define a family of rewriting rules as follows.

\begin{definition} \label{def:rules}
  Let $\hyper{H}$ be a connected hypergraph. 
  Let $K$ be a connected subset of $H$, and let $X \subseteq K$ be such that $\mathbb{K},X \leadsto U_1,\ldots U_n$.
  Let $Y \subseteq U_i$ and suppose that $\mathbb{K}_i\setminus Y \leadsto V_1,\ldots, V_k$. 
  Then, we define the \defn{rewriting rule}
  $$(X,K)(U_1,\ldots, U_{i-1},(Y,U_i)(V_1,\ldots, V_k),U_{i+1},\ldots,U_n)$$
  $$ \longrightarrow (X\cup Y,K)(U_1,\ldots,U_{i-1},V_1,\ldots, V_k,U_{i+1},\ldots,U_n)$$
 if $\max(Y)<\min(X)$, and 
  $$(X\cup Y,K)(U_1,\ldots,U_{i-1},V_1,\ldots, V_k,U_{i+1},\ldots,U_n)$$
  $$\longrightarrow (X,K)(U_1,\ldots, U_{i-1},(Y,U_i)(V_1,\ldots, V_k),U_{i+1},\ldots,U_n)$$
  if $\max(X)<\min(Y)$.
  We denote this set of rules by $R_\hyper{H}$.
\end{definition} 
We note that in the above definition $(X,K), (Y,U_i),\ldots$ are function symbols, while $U_1,\ldots,V_1,\ldots$ are variables.
It is clear that these are well-defined rewriting rules: the term on the left is never a variable, and the variables on both sides are the same.
Their closed instantiations consist in replacing the variables by actual constructs, via the identification provided by~ \cref{l:bijection-terms}.

Recall the covering relation $\prec$ from~ \cref{subface-relation}.
\begin{lemma} 
  The closed instantiations in context of the rewriting rules $R_\hyper{H}$ admit the following description.
  Let $S,T$ be two constructs such that $S \prec T$.  
  Then we have
 $$\begin{array}{lll}
    S \to T &  \mathrm{if} & \max(Y)<\min(X), \\
    T \to S & \mathrm{if} & \max(X)<\min(Y).
  \end{array}$$
\end{lemma} 

\begin{proof}
  Restricting our attention to closed terms, and using \cref{l:bijection-terms} gives the result. 
  Concretely, one just needs to set $K\eqdef \supp(\occ{S}{X})$ and $K_i\eqdef \supp(\occ{S}{Y})$ in \cref{def:rules}, and map each variable to a construct by some substitution $\sigma$.
\end{proof}

Behind the scene, the two clauses are not as symmetric as they seem to be. 
Procedurally speaking, in the first case, when moving from $S$ to $T$, there is nothing else to check than the condition $ \max(Y)<\min(X)$, while in the second case, when moving from $T$ to $S$, one has first to decide on a splitting of a node $Z$ of $T$ as some $X\cup Y$ in such a way that $Y$ is connected in $\supp(\occ{T}{Z})\setminus X$ and that the condition $\max(X)<\min(Y)$ holds.

\begin{rem}
  \label{rem:facialweak}
  This can be seen as the definition of a preorder on the set $\mathcal{A}(\hyper{H})$ of constructs of $\hyper{H}$, distinct from the face relation. 
  It can be easily checked that it coincides with the \emph{facial weak order} on the faces of permutahedra \cite{KrobLatapyNovelliPhanSchwer,PalaciosRonco,DermenjianHohlwegPilaud} and the \emph{generalized Tamari order}  on the faces of associahedra \cite{Ronco-Tamari}. 
  Indeed, via the identifications of the respective constructs with planar trees and surjections recalled in \cref{ss:hypergraph-polytopes}, one gets exactly the relations defined in \cite[Definition 30 and Lemma 17]{PalaciosRonco}.
  Is this preorder, for an arbitrary nestohedron, a partial order? 
  The answer is positive, and a simple proof of this fact will be available soon \cite{RSON}.
\end{rem}


\subsection{Rewriting on constructions}
\label{ss:rewriting-constructions}

Restricting our attention to constructions, it is natural to adapt the term rewriting system  $(\Sigma_\hyper{H},R_\hyper{H})$.

\begin{definition}
  For a connected hypergraph $\hyper{H}$, we consider the \defn{constructions signature}~$\Sigma_\hyper{H}^c$ defined by the following data:
 \begin{itemize}
    \item The variables are the connected subsets of $H$, that is $V\eqdef \{ X \subseteq H \ | \ \hyper{H}_{X} \text{ is connected}\}$. 
    \item The function symbols are ordered pairs $(x,Y)$ where $Y$ is a connected subset of $H$ and where $x\in Y$.
    $$F\eqdef \{(x,Y) \ | \ x \in Y \subseteq H, \ \hyper{H}_Y \text{ is connected}\}.$$
    \item The sorts are the connected subsets of $H$, i.e., 
    $S\eqdef \{ X \subseteq H \ | \ \hyper{H}_X \text{ is connected}\}$.
    \item For $(x,Y) \in F$, we define $\ari(x,Y)$ as the number of connected components of~$\hyper{H}_Y\setminus {x}$.
    \item Variables $X \in V$ are their own sort $\outsort(X)\eqdef X$, while function symbols~$(x,Y) \in F$ are of sort $\outsort(x,Y)\eqdef Y$.
    \item For function symbols $(x,Y) \in F$ such that $\hyper{H}_Y,x \leadsto Y_1,\ldots,Y_n$, and for $1 \leq i \leq n$, we define $\insort((x,Y),i)\eqdef Y_i$.
  \end{itemize}
\end{definition}

\begin{lemma} 
  \label{l:bijection-constructions}
  There is a bijection between the set of closed terms of  sort $H$ over $\Sigma_\hyper{H}^c$ and the set of constructions of $\hyper{H}$.
\end{lemma}

\begin{proof}
  The proof is similar to that of \cref{l:bijection-terms}.
\end{proof}

The rewriting rules that we next define are obtained by joining together the rules in \cref{def:rules} (see Remark \ref{consruction-join-constructs}). 

\begin{definition} 
  \label{def:rules-2}
  Let $\hyper{H}$ be a connected hypergraph. 
  Let $K$ be a connected subset of $H$, and let $x,y \in K$ be such that
  $$\mathbb{K},\{x,y\} \leadsto U_1,\ldots,U_\ell,V_1,\ldots,V_m,W_1,\ldots,W_n,$$
  where $U_1,\ldots,U_\ell$ are the connected components of $\hyper{K} \setminus x$ which do not contain $y$, $W_1,\ldots,W_n$ are the connected components of $\hyper{K} \setminus y$ which do not contain $x$, and $V_1,\ldots,V_m$ are the remaining connected components. 
  Let $K_y$ (resp.\ $K_x$) denote the connected component of $\mathbb{K} \setminus x$ (resp. $\mathbb{K} \setminus y$) which contains $y$ (resp.\ $x$).
  Then we define the \defn{rewriting rule}
  $$(x,K)(U_1,\ldots, U_{\ell},(y,K_y)(V_1,\ldots,V_m,W_1,\ldots,W_n))$$
  $$ \longrightarrow (y,K)(W_1,\ldots,W_n,(x,K_x)(V_1,\ldots,V_m,U_1,\ldots,U_\ell))$$
  whenever $x < y$. 
  We denote this set of rules by $R_\hyper{H}^c$.
\end{definition} 

Once again, it is clear that these are well-defined rewriting rules: the term on the left is never a variable, and the variables on both sides are the same.
Their closed instantiations in context define rewriting steps on the set of constructions of a given hypergraph, i.e., on the vertices of a given nestohedron.  

We also note that the two rewriting systems $(\Sigma_\hyper{H},R_\hyper{H})$ and $(\Sigma_\hyper{H}^c,R_\hyper{H}^c)$ are linear.

\begin{rem} \label{consruction-join-constructs}
We can see that $(\Sigma_\hyper{H}^c,R_\hyper{H}^c)$ can be simulated in $(\Sigma_\hyper{H},R_\hyper{H})$, in the sense that any rewriting step in $R_\hyper{H}^c$ is obtained by concatenating two rewriting steps in $R_\hyper{H}$ :
 $$\begin{array}{l}(\set{x},K)(U_1,\ldots, U_{\ell},(\set{y},K_y)(V_1,\ldots,V_m,W_1,\ldots,W_n))\\
 \quad  \longrightarrow (\set{x,y},K)(U_1,\ldots,U_\ell,V_1,\ldots,V_m,W_1,\ldots,W_n)\\
 \quad
 \longrightarrow (\set{y},K)(W_1,\ldots,W_n,(\set{x},K_x)(V_1,\ldots,V_m,U_1,\ldots,U_\ell)).
 \end{array}$$
\end{rem}

\begin{lemma} 
  \label{l:instantiation-constructions}
  The closed instantiations in context of the rewriting rules $R_\hyper{H}^c$ admit the following description.
  Let $S,T$ be two distinct constructions which are both vertices of the same edge. 
  This means that there exists a node $x$ of $S$ such that $\occ{S}{x}=x(y(\cdots),\cdots)$, and that $T$ is obtained by replacing in $S$ the  subtree rooted at $x$ with $y(x(\cdots),\cdots)$. 
  Then we have
  $$\begin{array}{lll}
    S \to T &  \mathrm{if} & x < y. 
  \end{array}$$
\end{lemma} 

\begin{proof}
  Restricting our attention to closed terms and using \cref{l:bijection-constructions} gives the result. 
  Concretely, one just needs to set $K\eqdef \supp(\occ{S}{x})$ and $K_y\eqdef \supp(\occ{S}{y})$ in \cref{def:rules-2}, and map each variable to a construct by some substitution $\sigma$.
\end{proof}

We shall write $S \pre{x}{y} T$ to record that the  minimal  subtree of $S$ responsible for the reduction from $S$ to $T$ is $\occ{S}{x}$ and that the reduction concerns the son $y$ of $x$.
This defines a preorder $<$ on the set of consructions of an hypergraph $\hyper{H}$.
Is this preorder a partial order? 

It turns out that our definition of $<$ on constructions is equivalent to the definition of the \emph{flip relation} on maximal tubings of a graph associahedron given by Barnard and McConville in \cite{Barnard-McConville}, of which particular cases are considered in ~\cite{Forcey-Tamari}.
Their proof that the reflexive and transitive closure of the flip relation is a partial order extends readily to hypergraph polytopes.
Below, we recall and ``upgrade'' their proof, adapting it to our setting.

\begin{definition}
  \label{def:coordinate-vector}
  Given a construct $S$ of an ordered hypergraph $\hyper{H}$, its \defn{coordinate vector}
  $v^S=(\ldots,v^S_y,\ldots,v^S_x,\ldots) \in \mathbb{R}^{|H|}$, where the coordinates appear according to the increasing order of the elements of $H$, is defined by
  $$v^S_x\eqdef |\setc{e\in{\Sat}(\hyper{H})}{x\in e\inc \supp(\occ{S}{x})}|.$$
\end{definition}

\begin{proposition}[{\cite[Lem.~2.8]{Barnard-McConville}}] 
\label{p:flip-partial-order}
Let $\hyper{H}$ be an ordered connected hypergraph. 
The preorder generated by the flip relation $<$ defined above is a partial order, that is well-founded.
\end{proposition}

\begin{proof}
Let $S,T$ be as in \cref{l:instantiation-constructions}. 
We set $K\eqdef \supp(\occ{S}{x})=\supp(\occ{T}{y})$, $I\eqdef \supp(\occ{S}{y})$, and $J\eqdef \supp(\occ{T}{x})$.
Let us examine $v^S$ and $v^T$. 
One sees easily that they have the same coordinates in all positions other than $x$ and $y$. 
We have, by definition
$$ v^S_x  = |\setc{e\in{\Sat}(\hyper{H})}{x\in e\inc K}| \quad \text{and} \quad v^T_x  = |\setc{e\in{\Sat}(\hyper{H})}{x\in e\inc J}|, $$
$$ v^S_y  = |\setc{e\in{\Sat}(\hyper{H})}{y\in e\inc I}| \quad \text{and} \quad v^T_y  = |\setc{e\in{\Sat}(\hyper{H})}{y\in e\inc K}|.$$
We next claim that the following equality holds:
$$v^S_x  - v^T_x  = \lambda = 
v^T_y - v^S_y \quad\textrm{where}\: \lambda\eqdef  |\setc{e\in{\Sat}(\hyper{H})}{\set{x,y}\inc e\inc K}|.$$
We just prove the inclusion of sets
$$\left(\setc{e\in{\Sat}(\hyper{H})}{x\in e\inc K}\setminus\setc{e\in{\Sat}(\hyper{H})}{x\in e\inc J}\right)
\inc\setc{e\in{\Sat}(\hyper{H})}{\set{x,y}\inc e\inc K}.$$
Suppose that $e$ is connected, $x\in e$, $e\inc K$ and $e\not\inc J$, and $y\not\in e$. Then $y$ has to lie entirely inside one of the connected components of $\restrH{K}{y}$, which has to be $J$ since $x\in e$, contradicting~$e\not\inc J$.
We thus have
$v^S - v^T = (0,\ldots,-\lambda,\ldots,\lambda,\ldots)$, where $\lambda$ is a positive integer.
Consider now an arbitrary vector $\mu=(\ldots,\mu_y,\ldots,\mu_x,\ldots)$ such that $\mu_{\bullet}:H\rightarrow\mathbb{R}$ is strictly decreasing, and consider the linear functional $\overline{\mu}\eqdef \langle -,\mu\rangle$. 
Then we have 
$\overline{\mu}(v^S)- \overline{\mu}(v^T)=\lambda(\mu_x-\mu_y)<0$.  The well-foundedness of $\to$ then follows from the finiteness of the set of constructions. In turn, well-foundedness 
prevents to create cycles when composing flips, concluding the proof.
\end{proof}

\begin{thm}
  \label{thm:termination}
  The rewriting system $(\Sigma_\hyper{H}^c,R_\hyper{H}^c)$ is terminating.
\end{thm}

\begin{proof} 
  We observe that there is a closed term for every sort $X$: just take of 0-ary function symbol $(X,X)$ (cf.~ \cref{open-to-closed}). 
  Therefore, by~ \cref{closed-termination}, it suffices to show that termination holds for closed terms. 
  This follows immediately from \cref{l:instantiation-constructions,p:flip-partial-order}.
\end{proof}

But there is more to it. 
It turns out that the map $v^{\bullet}$ from constructions to $\mathbb{R}^{|H|}$
in \cref{def:coordinate-vector} has a geometric significance.  
Let $\Delta^{n-1}$ be the standard $(n-1)$-dimensional \defn{standard simplex} in $\mathbb{R}^n$, the convex hull of the basis vectors $e_1,\ldots,e_n$. 
Each non-empty subset $I \subseteq \set{1,\ldots,n}$ determines a face $\Delta^I$ of $\Delta^{n-1}$, the convex hull of $\setc{e_i}{i\in I}$. 

\begin{definition}
  Let $\hyper{H}$ be a connected hypergraph such that $H=\set{1,\ldots,n}$. 
  The \defn{Postnikov realization} of the hypergraph polytope associated to $\hyper{H}$ is the Minkowski sum
  $$P_{\hyper{H}}\eqdef \sum_{E\in{\Sat}(\hyper{H})} \Delta^E.$$
\end{definition}

\begin{proposition}[{\cite[Prop.~7.9]{P09}}] 
  \label{Postnikov-correspondence}
  Let $\hyper{H}$ be an ordered connected hypergraph.
  The map $v^{\bullet}$ is a bijection from the set of constructions of $\hyper{H}$ to the set of vertices of $P_{\hyper{H}}$.
\end{proposition}

Constructions correspond to \emph{maximal nested sets} in the terminology of \cite{P09}.

Recall that an \defn{orientation vector} of a polytope $P\subset \R^n$ is a vector $\mu \in \R^n$ which is not perpendicular to any edge of $P$. 

\begin{corollary} 
  \label{Tamari-orientation-vector}
  Any vector $\mu$ with strictly decreasing coordinates is an orientation vector for $P_{\hyper{H}}$.
\end{corollary} 

\begin{proof}
The statement follows immediately from reading the proof of \cref{p:flip-partial-order} in the light of \cref{Postnikov-correspondence}.
\end{proof}

Our presentation is anachronical, since the vectors $v^S$ were preexisting to their use by Barnard and McConville.
But it stresses the fact that the proof of termination in~ \cref{p:flip-partial-order} is purely combinatorial and does not rely on the existence of a geometric realization.

%


\subsection{Critical pairs and confluence}
\label{ss:critical}

We next examine the orientation induced on the $X$-faces of $\hyper{H}$, for some $X=\set{x_1,x_2,x_3}\inc H$. 
Depending on the total order chosen on $H$, each of the four shapes of type (B) from~\cref{ss:typeB} gives rise to 6 possible local confluence diagrams. 
We list them below in schematic form (i.e., without the $\cdots$) for the quadrilateral shape (B2).
We overline (resp.\ underline) the minimum (resp.\ maximum) in the partial order.
$$\begin{array}{ccc}
\boxed{x_1>x_2>x_3} && \boxed{x_1>x_3>x_2}\\
&&\\
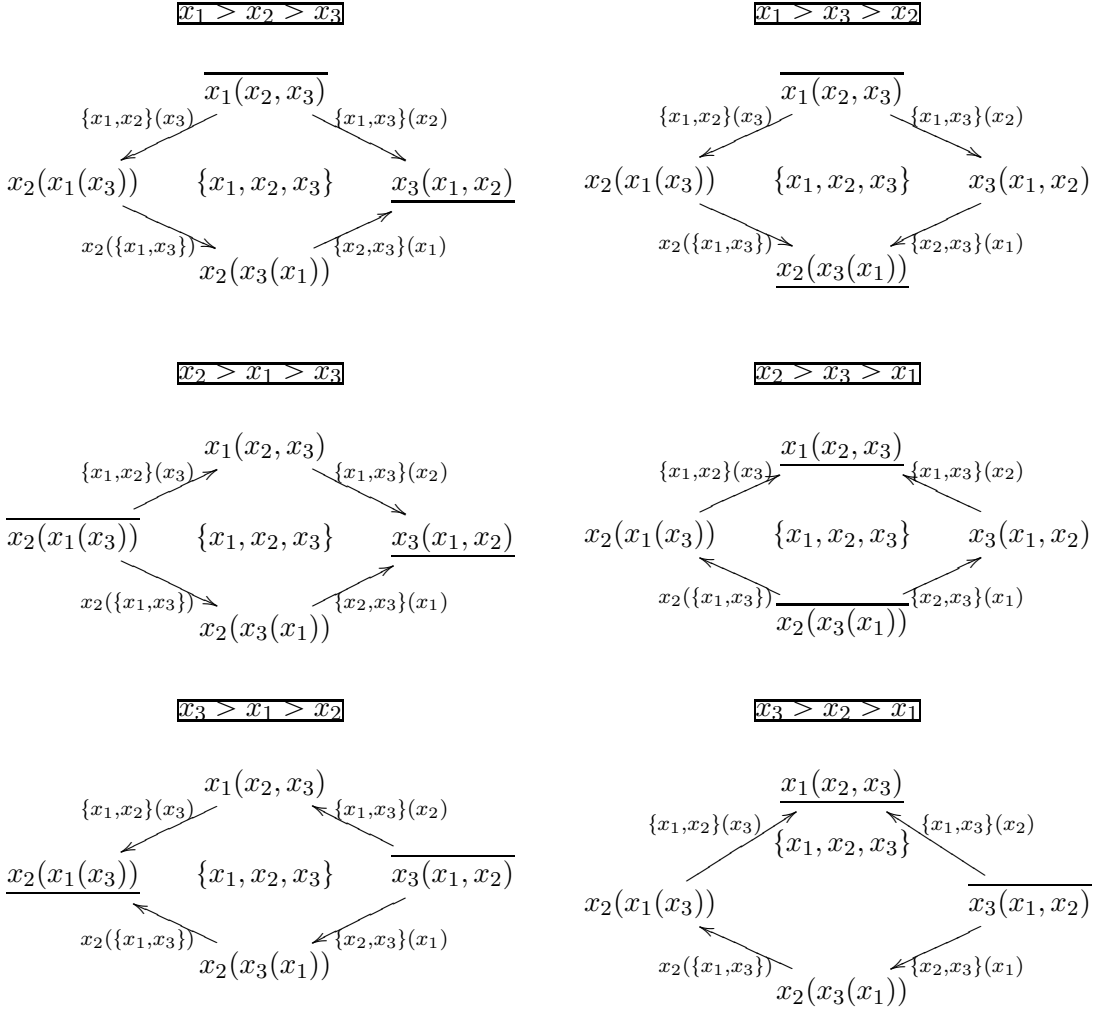
\xymatrix @-1.65pc {&& \overline{x_1(x_2,x_3)} \ar @{->}[ddll]_{\set{x_1,x_2}(x_3)} \ar @{->}[ddrr]^{\set{x_1,x_3}(x_2)}&& \\
&&&&\\
x_2(x_1(x_3)) \ar @{->}[ddrr]_{x_2(\set{x_1,x_3})}&& \set{x_1,x_2,x_3}  && \underline{x_3(x_1,x_2)}\ar @{<-}[ddll]^{\set{x_2,x_3}(x_1)}\\
&&&&\\
&&  x_2(x_3(x_1))&&}
&&

\xymatrix @-1.65pc {&& \overline{x_1(x_2,x_3)} \ar @{->}[ddll]_{\set{x_1,x_2}(x_3)} \ar @{->}[ddrr]^{\set{x_1,x_3}(x_2)}&& \\
&&&&\\
x_2(x_1(x_3)) \ar @{->}[ddrr]_{x_2(\set{x_1,x_3})}&& \set{x_1,x_2,x_3}  && x_3(x_1,x_2)\ar @{->}[ddll]^{\set{x_2,x_3}(x_1)}\\
&&&&\\
&&  \underline{x_2(x_3(x_1))}&&}
\end{array}
$$

\medskip
$$\begin{array}{ccc}
\boxed{x_2>x_1>x_3} && \boxed{x_2>x_3>x_1}\\
&&\\
\xymatrix @-1.65pc {&&x_1(x_2,x_3) \ar @{<-}[ddll]_{\set{x_1,x_2}(x_3)} \ar @{->}[ddrr]^{\set{x_1,x_3}(x_2)}&& \\
&&&&\\
\overline{x_2(x_1(x_3))} \ar @{->}[ddrr]_{x_2(\set{x_1,x_3})}&& \set{x_1,x_2,x_3}  && \underline{x_3(x_1,x_2)}\ar @{<-}[ddll]^{\set{x_2,x_3}(x_1)}\\
&&&&\\
&& x_2(x_3(x_1))&&}
&&

\xymatrix @-1.65pc{&&  \underline{x_1(x_2,x_3)} \ar @{<-}[ddll]_{\set{x_1,x_2}(x_3)} \ar @{<-}[ddrr]^{\set{x_1,x_3}(x_2)}&& \\
&&&&\\
x_2(x_1(x_3)) \ar @{<-}[ddrr]_{x_2(\set{x_1,x_3})}&& \set{x_1,x_2,x_3}  && x_3(x_1,x_2)\ar @{<-}[ddll]^{\set{x_2,x_3}(x_1)}\\
&&&&\\
&&   \overline{x_2(x_3(x_1))}&&}
 \end{array}
$$

$$\begin{array}{ccc}
\boxed{x_3>x_1>x_2} && \boxed{x_3>x_2>x_1}\\
&&\\
\xymatrix @-1.65pc {&& x_1(x_2,x_3) \ar @{->}[ddll]_{\set{x_1,x_2}(x_3)} \ar @{<-}[ddrr]^{\set{x_1,x_3}(x_2)}&& \\
&&&&\\
\underline{x_2(x_1(x_3))} \ar @{<-}[ddrr]_{x_2(\set{x_1,x_3})}&& \set{x_1,x_2,x_3}  && \overline{x_3(x_1,x_2)}\ar @{->}[ddll]^{\set{x_2,x_3}(x_1)}\\
&&&&\\
&&  x_2(x_3(x_1))&&}
&&

\xymatrix @-1.65pc {&& \underline{x_1(x_2,x_3)} \ar @{<-}[ddll]_{\set{x_1,x_2}(x_3)} \ar @{<-}[ddrr]^{\set{x_1,x_3}(x_2)}&& \\
&&\set{x_1,x_2,x_3}&&\\
x_2(x_1(x_3)) \ar @{<-}[ddrr]_{x_2(\set{x_1,x_3})}&&   && \overline{x_3(x_1,x_2)}\ar @{->}[ddll]^{\set{x_2,x_3}(x_1)}\\
&&&&\\
&&  x_2(x_3(x_1))&&}
  \end{array}
$$
Each of these local confluence diagrams stems from a critical pair. 
For example, in the first diagram, we can see the minimal overlapping between applying rewriting to the parts $x_1(-,x_3)$ and $x_1(x_2,-)$ of the construct $x_1(x_2,x_3)$.   Let us make precise what we mean by ``parts''.
Calling $\hyper{K}$ the hypergraph underlying the poset (B2), and $t$ the term corresponding to the construction $x_1(x_2,x_3)$, we have
(cf.~ \cref{open-to-closed})
$t=\sigma_1(t_1)$ (resp.  $t=\sigma_2(t_2)$), where $t_1 \eqdef (x_1,K)(\set{x_2},(x_3,\set{x_3}))$ and $\sigma_1(x_2)\eqdef (x_2,\set{x_2})$ (resp.
$t_2\eqdef (x_1,K)((x_2,\set{x_2}),\set{x_3})$ and $\sigma_2(x_3)\eqdef (x_3,\set{x_3}$)). This allows us to see 
the edges from $t_1$  to $x_2(x_1(x_3))$  and to  $x_3(x_1,x_2)$ as instantiations of the rewriting rules of~ \cref{def:rules-2}. We shall make this formal in \cref{face-flip,thm:critical-pairs}. 

Recall that a function $f$ preserves (resp. reflects) a relation  ${\cal R}$ if for all $x,y$, $x\:{\cal R}\:y$ implies $f(x)\:{\cal R}\: f(y)$ (resp.  $f(x)\:{\cal R} \:f(y)$ implies
$x\:{\cal R}\:y$).

\begin{lemma} \label{face-flip}
  Let $\hyper{H}$ be an ordered connected hypergraph. 
  The bijections $\psi$ (resp.\ $\chi$) preserve and reflect the flip order $<$ on constructions (resp.\ the rewriting steps $\to$),
\end{lemma}
\begin{proof} 
  The proof is an easy variation on the proof of~\cref{instance-construct} (resp. of  \cref{l:bijection-open}).
\end{proof}

Recall from \cref{two-types} and~\cref{recollection-section} that $2$-faces of~$\hyper{H}$ can be of type (A) or (B), while local confluence diagrams of~$(\Sigma_\hyper{H}^c,R_\hyper{H}^c)$ can be of type (a1), (a2) or (b).

\begin{thm}
  \label{thm:critical-pairs} 
  The rewriting system $(\Sigma_\hyper{H}^c,R_\hyper{H}^c)$ is locally confluent. The local confluence diagrams originating from  closed terms of  sort $H$ are in one-to-one correspondence with the oriented 2-faces of $\hyper{H}$.
  More precisely, the 2-faces of type (A) provide the local confluence diagrams of type (a1) and (a2),  and the $X$-faces of type (B) provide all the confluence diagrams of type (b).
\end{thm}

\begin{proof}
  We use \cref{l:bijection-constructions,l:instantiation-constructions,face-flip} to work directly with constructions. 
Let us consider three constructions $S,T,U$ such that $S\pre{x}{y}T$ and $S \pre{u}{v} U$, with $(x,y)\neq(u,v)$.
  We have~$\occ{S}{x}=x(y(\ldots),\ldots)$ and $\occ{S}{u}=u(v(\ldots),\ldots)$. 
  There are two cases to consider:
  \begin{itemize}
  \item[(A)] $\set{x,y}\cap\set{u,v}=\emptyset$: then the two reductions do not overlap and we are in the situation in which $S,T,U$ fit in a 2-face of type (A).  One sees easily that we get local confluence diagrams of type (a1) (resp. of type (a2)) if the edge from $x$ to $y$ is disjoint from (resp. below or above) the edge  from $u$ to $v$ in $S$. 
  \item[(B)] $\set{x,y}\cap\set{u,v}\neq\emptyset$. 
  There are a priori four subcases:
  \begin{itemize}
  \item $x=u$: then $\occ{S}{x}=x(y(\ldots),v(\ldots),\ldots)$;
  \item $y=u$: then  $\occ{S}{x}=x(y(v(\ldots),\ldots),\ldots)$;
  \item $x=v$: up to permuting $T$ and $U$, this is the previous case;
  \item $y=v$: this would force $x=u$, contrary to our assumption.
  \end{itemize}
  This gives evidence that case (B) features the two (and only two) overlapping situations $x(y(\ldots),v(\ldots),\ldots)$ (with $x<y$ and $x<v$) and 
  $x(y(v(\ldots),\ldots),\ldots)$ (with $x<y<v$), and the four subcases  (in their oriented version as above) show how to complete the local confluence diagrams.  
\end{itemize}.
\end{proof}

\begin{thm}
  \label{thm:confluent}
  The rewriting system $(\Sigma_\hyper{H}^c ,R_\hyper{H}^c)$ is confluent and terminating.
\end{thm}

\begin{proof}
  By \cref{thm:termination} and \cref{thm:critical-pairs}, $(\Sigma_\hyper{H}^c,R_\hyper{H}^c)$ is terminating and locally confluent, and therefore confluent by~\cref{thm:Newman}.
\end{proof}

We note that the  proof of \cref{thm:critical-pairs} above  does not make use of the critical pair \cref{l:local-confluence-critical-pair}. But we can analyse critical pairs geometrically, as we show now.
Recall the functions $\psi$ and $\chi$ from 
 \cref{instance-construct,l:bijection-open}.
The following Lemma substantiates the view that $X$-faces are instantiations in context of their shape, and is the key to the analysis of critical pairs in \cref{critical-pairs-nestohedra} below. 
Its statement requires a bit of ``yoga''.
 
Our goal is to exhibit any closed term $s$ corresponding to a construct of an $X$-face $T$ as an instantiation in context  of an open term associated with the shape of that face.
For this we use the bijections $\chi$ and $\psi$. We first transform $s$ into a construct $S:=\chi(s)$. 
Then we use $\psi$ to get a construct $S':=\psi(S)$ of the same dimension in the shape of $T$. 
Finally, we use the inverse of $\chi$ to get an open term
$t':=\chi^{-1}(S')$. 
The next lemma states that our original~$t$ is an instance in context of $t'$.

\begin{lemma} 
  \label{instance-in-context}
Let $T$ be a $X$-face of a connected hypergraph $\hyper{H}$, and suppose that $\hyper{H},X\leadsto U_1,\ldots U_n$. 
Then, the composite $\xi\eqdef \chi^{-1}\circ \psi\circ \chi$ mapping terms over $\Sigma_{\recrestr{(\hyper{H}_{\supp(\occ{T}{X})})}{X}}$ to closed terms over $\Sigma_{\hyper{H}}$ admits the following description: 
\begin{itemize}
  \item Writing $\chi^{-1}(T)=C[(X,\supp(\occ{T}{X}))(t_1,\ldots t_n)]$ and defining the substitution $\sigma$ by $\sigma(U_i)\eqdef t_i$, then we have $\xi(t')=C[\sigma(t')]$, for all $t'$ in the domain of $\xi$.
\end{itemize}

\end{lemma}

\begin{proof}
As is quite common in matters involving syntax, the difficulty lies more in formulating the statement than in proving it. The proof consists in carefully tracking the successive transformations. 
The key observation is that the construction of the inverse $\phi$ of $\psi$ in the proof of \cref{instance-construct} ``secretly" performs an instantiation. 
The details are left to the reader.
\end{proof}

\begin{proposition}
\label{critical-pairs-nestohedra}
The critical confluence diagrams of the rewriting system $(\Sigma_\hyper{H}^c,R_\hyper{H}^c)$  are provided by the maximal  faces of all $\recrestr{(\hyper{H}_Y)}{X}$, for $X\inc Y \inc H$ with $Y$ connected and $X$ of cardinality 3.
\end{proposition}
\begin{proof}
This follows from \cref{instance-in-context}.
\end{proof}

\begin{rem}
  \label{rem:coherence}
A consequence of a term rewriting system  being confluent and terminating is that it is \defn{coherent}:
every two parallel  sequences
 $s\leftrightarrow s_1 \leftrightarrow \ldots \leftrightarrow s_m\to t$ and $s\leftrightarrow s'_1 \leftrightarrow \ldots \leftrightarrow s'_n\leftrightarrow t$, where $\leftrightarrow$ is the symmetric closure of $\to$,  can be proved equal by successive transformations replacing a part of a  sequence by a ``complementary''  sequence forming with it the (non-oriented) border of a local confluence diagram. 
 This statement can be proved by following Huet's steps sketched in the introduction.  
 A nice and detailed exposition can be found in~\cite{Beke-Knuth}, see also~\cite{GM-FDT} in the 
 setting of word rewriting.
 Alternatively, via the dictionary  established in this section between rewriting and nestohedra, this result, for $(\Sigma_\hyper{H}^c ,R_\hyper{H}^c)$, falls out as a special case of our combinatorial coherence theorem in \cite[Thm~1.4 \& Prop.~1.7]{CLA1}. 
 We observe here that since nestohedra are simple polytopes~\cite[Sec.~9]{DP-HP}, the proof of the instrumental Lemma 1.3 in \cite{CLA1} can be simplified, as the simplicity assumption implies that the case (3) considered in its proof does not apply. 
 This makes our combinatorial proof in \cite{CLA1} even more akin to the rewriting proof.
 \end{rem}


\section{Contextual families of nestohedra} 
\label{s:contextual}

In this section, we define a subclass of the class of  nestohedra, which, as we shall argue, brings us even closer to
our favourite categorical coherence results. We begin by giving some intuition.

\subsection{Discussion} \label{contextual-discussion}

In the previous section, we have established coherence (in our polytopal sense) using standard term rewriting methods. So far so good. But  it can be argued that constructs are more economical than the terms on the signatures that we have introduced, and that, being themselves trees, they ``look like" terms. 
Can we somehow reproduce our discussion of rewriting, and in particular of critical pairs, directly on constructs?

An obstacle is that the bijection $\chi$ of~ \cref{l:bijection-terms} strips off useful information on the support of all subconstructions (i.e.,  subtrees) of the constructions to be ``rewritten''. 
Taking~ \cref{l:instantiation-constructions} seriously, one can be tempted to consider $x(y(\cdots),\cdots)< y(x(\cdots),\cdots)$ as a rewriting rule (for $x<y$), but filling in the $\cdots$ requires knowing the support $K$ of $\occ{S}{x}$. 
So, at the price of reconstructing $K$ from the inductive definition of $S$, this description is acceptable at the level of rewriting steps. But for the study of critical confluence diagrams, say   $S < x_2(x_1(x_3,\cdots),\cdots),\cdots)$ and $S < x_3(x_1(\cdots),x_2(\cdots),\cdots)$, for $S\eqdef  x_1(x_2(\cdots),x_3(\cdots),\cdots)$, can we be sure that the shape of the local confluence diagram will be independent of the support of $\occ{S}{x_1}$?  

In view of~ \cref{instance-construct-general}, this shape is
$\recrestr{(\hyper{H}_{\supp(\occ{T}{x_1})})}{\set{x_1,x_2,x_3}}$. 
It would be nicer if the shape was $\recrestr{\hyper{H}}{\set{x_1,x_2,x_3}}$, because then the critical confluence diagram would be ``context-independent'' (for this informal notion of rewriting), i.e., all $\set{x_1,x_2,x_3}$-faces would have the same shape.
This motivates the definition of contextual hypergraph below. 


\subsection{Contextual nestohedra}

\begin{definition} 
A connected hypergraph $\hyper{H}$ is \defn{contextual} if, for all connected subsets $Y\inc H$ of cardinality $|Y|\geq 3$, and for all $3$-elements subsets $X=\{x,y,z\} \subseteq Y$, one of the following equivalent conditions is satisfied:
\begin{enumerate}
\item
$\begin{array}{lll}
  \xyz{x}{{\hyper{H}_Y}}{\set{y,z}} & \Leftrightarrow & \xyz{x}{\hyper{H}}{\set{y,z}}
  \end{array}$,
\item
$\hyper{H}_{\cap X} = (\hyper{H}_Y)_{\cap X}$.
\end{enumerate}
\end{definition}
That these conditions are equivalent   is a direct consequence of~\cref{xyz-reconnected}.
We first give two examples of hypergraphs that are \emph{not} contextual.


\begin{example} 
  \label{non-contextual-1}
Consider the hypergraph 
\[
  \hyper{H}\eqdef  \set{\set{x},\set{y},\set{z},\set{u},\set{x,y,z}, \set{x,u,z}},
  \]
the set $X\eqdef \set{x,y,z}$ and the two $X$-faces $S\eqdef u(X)$ and $T\eqdef X(u)$. 
Then $\occ{S}{X}$ is a construct of $\hyper{K}\eqdef \restrH{H}{\set{u}}$ while $\occ{T}{X}=T$ is a construct of $\hyper{H}$.
But we have $\xyz{y}{\hyper{K}}{\set{x}\!,\!\set{z}}$ and $\xyz{y}{\hyper{H}}{\set{x,z}}$, which implies that $S$ is a triangle while $T$ is a quadrilateral, as
$\recrestr{\hyper{K}}{\set{x,y,z}} = \hyper{K}  =  \set{\set{x},\set{y},\set{z},\set{x,y,z}}$ and $\recrestr{\hyper{H}}{\set{x,y,z}}  =  \set{\set{x},\set{y},\set{z},\set{u},\set{x,z}\set{x,y,z}}$. Therefore $\hyper{H}$ is not contextual.
\end{example}
  
\begin{example} 
  \label{non-contextual-2}
Consider the graph 
$$\set{\set{x},\set{y},\set{z},\set{u},\set{x,y}, \set{y,z}, \set{x,u}, \set{u,z}}.$$
Then exactly the same data as in~ \cref{non-contextual-1} provide evidence that this graph, whose realization is the three-dimensional cyclohedron, is not contextual. 
\end{example}


The following Proposition allows us to see all $X$-faces as ``instantiations in context'' of~$\recrestr{\hyper{H}}{X}$.

\begin{proposition} \label{situation-construct}
Let $\hyper{H}$ be a contextual hypergraph.
If $X$ is a subset of $H$ such that $|X|=3$ and $T$ is an $X$-face of $\hyper{H}$, then the poset of faces of $T$ is isomorphic to the poset of faces of~$\recrestr{\hyper{H}}{X}$.
\end{proposition}

\begin{proof}  This is a direct consequence of~ \cref{instance-construct-general}.
\end{proof}


\subsection{Contextual families}

Motivated by the examples presented in \cref{ss:hypergraph-polytopes}
and their associated categorical coherence theorems listed in \cref{table:contextual-hyper}, we define now the notion of a contextual \emph{family} of nestohedra.

Identifying a hypergraph $\hyper{H}$ with the maximal construct $T$ of $({\cal A}(\hyper{H}),\preceq)$, we say that $\hyper{H}$ has \defn{dimension} $\dim T$ (cf.  \cref{two-types}).
For a family of hypergraphs $\calH$, we denote by $\calH(n)$ the subset of hypergraphs of dimension $n \geq 0$.

We will consider families of ordered hypergraphs. 
Note that when $\hyper{H}$ is ordered, all the restrictions $\hyper{H}_X$ and reconnected restrictions $\hyper{H}_{\cap X}$ are naturally ordered hypergraphs.

\begin{definition}
  \label{def:contextual-family}
    A family $\calH$ of ordered hypergraphs is \defn{contextual} if 
    \begin{enumerate}
      \item any ordered hypergraph $\hyper{H} \in \calH$ is contextual,
      \item for any $\hyper{H} \in \calH$ and any $X \subseteq H$, all the connected components of $\hyper{H}\setminus X$ are in $\calH$,
      \item we have $\{\hyper{H}_{\cap X} \ | \ X \subset H, |X|=3, \hyper{H} \in \calH \} \subseteq \calH(2)$.
    \end{enumerate}
\end{definition}
As for point (3) of this definition, we note that a reconnected restriction of a contextual hypergraph is contextual.

The term rewriting systems from \cref{ss:rewriting-constructs,ss:rewriting-constructions} can be adapted to a rewriting system on \emph{all} hypergraphs of $\calH$.
We shall focus on the constructions rewriting system. 

\begin{definition}
  For a contextual family of hypergraphs $\calH$, we consider the \defn{constructions signature}~$\Sigma_\calH^c$ defined by the following data:
  \begin{itemize}
    \item Variables and sorts are elements of $\calH$.
    \item Function symbols are ordered pairs $(x,\hyper{H})$ where $\hyper{H} \in \calH$ and $x\in H$:
    $$F\eqdef \{(x,\hyper{H}) \ | \ x \in H, \ \hyper{H} \in \calH \}.$$
    \item For $(x,\hyper{H}) \in F$, we define $\ari(x,\hyper{H})$ as the number of connected components of~$\hyper{H} \setminus {x}$.
    \item Variables $\hyper{H} \in V$ are their own sort $\outsort(\hyper{H})\eqdef \hyper{H}$, while function symbols~$(x,\hyper{H}) \in F$ are of sort $\outsort(x,\hyper{H})\eqdef \hyper{H}$.
    \item For function symbols $(x,\hyper{H}) \in F$ such that $\hyper{H},x \leadsto H_1,\ldots,H_n$, and for $1 \leq i \leq n$, we define $\insort((x,\hyper{H}),i)\eqdef \hyper{H}_i$.
  \end{itemize}
\end{definition}

It follows from~\cref{def:contextual-family} and the fact that the restriction of a contextual hypergraph is contextual, that this signature is well-defined.
Moreover, it is straightforward to adapt \cref{l:bijection-constructions} and \cref{def:rules-2} to obtain a term rewriting system $(\Sigma_\calH,R_\calH)$ on the constructions of $\calH$. 

From \cref{thm:critical-pairs}, we have that all local confluence diagrams for overlapping pairs of $(\Sigma_\calH,R_\calH)$ correspond to some $X$-face of some $\hyper{H} \in \calH$.
The fact that $\calH$ is contextual imposes an additional uniformity constraint on these diagrams.

\begin{thm}
Let $\calH$ be a contextual family of ordered hypergraphs.
For any $\hyper{H} \in \calH$ and subset $X \subseteq H$ with $|X|=3$, all the $X$-faces and hence their associated  overlapping local confluence diagrams 
have the same shape $\hyper{H}_{\cap X}$. 
\end{thm}

\begin{proof} 
This is a direct consequence of~\cref{situation-construct}.
\end{proof}

Mimicking~\cref{thm:confluent}, we get that $(\Sigma_\calH,R_\calH)$ is confluent and terminating.
Moreover, by virtue of Condition (3) in \cref{def:contextual-family}, all the critical confluence diagrams of $(\Sigma_\calH,R_\calH)$ are in~$\calH(2)$.
We argue that these diagrams should be called \emph{coherence conditions}, in view of \cref{rem:coherence} as well as of the following examples of contextual families and their coherence theorems. 


\subsection{Examples}
\label{ss:examples}

We call \defn{contextual graph-associahedra} (resp.\ \defn{contextual nestohedra}) the hypergraph polytopes whose underlying hypergraph is a contextual (hyper)graph.
Here, we include a copy of each (hyper)graph for each possible total order on its vertices.
Recall that simplices, cubes, associahedra, permutahedra and operahedra were introduced in \cref{ss:hypergraph-polytopes}.

\begin{samepage}
  \begin{thm}
    \label{thm:examples}
    The following families of hypergraph polytopes are contextual:
    \begin{enumerate}[label=(\alph*)]
      \item simplices,
      \item cubes,
      \item associahedra,
      \item permutahedra,
      \item operahedra,
      \item contextual graph-associahedra,
      \item contextual nestohedra.
    \end{enumerate}
  \end{thm}
\end{samepage}

\begin{proof}
  Let us proceed one family at a time.
  For each one, we check conditions (1)-(3) in \cref{def:contextual-family}.
  We consider sets of vertices to be $H=\{1,\ldots,n\}$.
  \begin{enumerate}[label=(\alph*)]
    \item Conditions (1)-(3) follow easily from the fact that hyperedges of simplices are all either singletons or the maximal hyperedge.
    \item We first prove condition (1). 
    Note $\hyper{C}_n$ is saturated, and that $(\hyper{C}_n)_{\set{1,\ldots,m}}=\hyper{C}_m$ if $m\leq n$.
    So we have to check that for all $m\leq n$ and all $i,j,k\leq m$, we have $\xyz{k}{{\hyper{C}_n}}{\set{i,j}}$ iff
    $\xyz{k}{{\hyper{C}_m}}{\set{i,j}}$, which follows immediately from the observation that for all $p\geq m$ we have
    $\xyz{k}{{\hyper{C}_p}}{\set{i,j}}$ iff $i<k$ and $j<k$.
    For conditions (2) and (3), it suffices to observe that the connected components of $\hyper{C}_n\setminus X$, for some $X$, are all cubes $\hyper{C}_m$ with $m<n$, and that reconnected restrictions of cubes are cubes.
    \item[(c)-(e)] Conditions (1) and (2) follow from the fact that any connected subgraph of a linear (resp. complete, clawfree block) graph is a linear (resp. complete, clawfree block) graph. 
    Condition (3) follows from the fact that any reconnected complement of a subset in a linear (resp. complete, clawfree block) graph is a linear (resp. complete, clawfree block) graph.
    \item[(f)-(g)] This is immediate from the definitions.
  \end{enumerate}
\end{proof}

\begin{rem}
  Note that contextual (hyper)graphs do not contain all graph-associahedra.
  For instance, we have seen in \cref{non-contextual-2} that the cyclohedra are not contextual. 
  It would be interesting to characterize combinatorially contextual (hyper)graphs.
\end{rem}


\subsection{Categorical coherence}
\label{ss:coherence}

In this final section, we recall the categorical coherence theorems associated with associahedra and operahedra, and conjecture one for permutahedra.


\subsubsection{Associahedra} \label{Huet-correspondence}
Recall that the scene is the data of a category $\mathbf C$, a bifunctor $\otimes:\mathbf{C}^2\rightarrow \mathbf C$ and a natural iso $\alpha$ from the functor
$(X,Y,Z)\mapsto (X\otimes Y)\otimes Z$ to the functor  $(X,Y,Z)\mapsto X \otimes (Y\otimes Z)$. 
Mac Lane's coherence theorem states that for any two functors $F,G$ from $\mathbf{C}^{n+1}$ to $\mathbf{C}$ arising from $n$ iterations of $\otimes$, any two  natural transformations $\lambda_1,\lambda_2$ from $F$ to $G$  ``written using $\alpha$ or its inverse'' are equal, provided the statement holds in the following special case, called  {\em coherence condition}: 
\begin{itemize}
\item $F\eqdef (X,Y,Z,U)\mapsto ((X\otimes Y)\otimes Z)\otimes U,$ 
\item $G\eqdef (X,Y,Z,U)\mapsto X\otimes (Y\otimes (Z\otimes U)),$ 
\item $\lambda_1\eqdef (X\otimes\alpha_{Y,Z,U})\circ\alpha_{X,Y\otimes Z,U} \circ (\alpha_{X,Y,Z}\otimes U),$ and
\item $\lambda_2\eqdef  \alpha_{X,Y,Z\otimes U}\circ \alpha_{X\otimes Y,Z,U},$ 
\end{itemize}
i.e. provided the following diagram (Mac Lane's pentagon) commutes:
\begin{center}
\vspace{-.5cm}
$$
 \xymatrix @-1.65pc {&& ((X\otimes Y)\otimes Z)\otimes U \ar @{->}[ddll]^{\alpha_{X,Y,Z}\otimes U} \ar @{->}[dddrr]^{\alpha_{X\otimes Y,Z,U}}&& \\
 &&&&\\
 (X\otimes (Y\otimes Z))\otimes U  \ar @{->}[dd]^{\alpha_{X,Y\otimes Z,U} }&&   && \\
 &&&&(X\otimes Y)\otimes (Z\otimes U) \ar @{->}[dddll]_{\alpha_{X,Y,Z\otimes U}}\\
 X\otimes ((Y\otimes Z)\otimes U) \ar @{->}[ddrr]_{X\otimes\alpha_{Y,Z,U}} &  &\\
 &&&&\\
 &&X\otimes (Y\otimes (Z\otimes U))&&}
$$
\end{center}

Via the Huet correspondence \cite{Huet-notes-cat}, the annotated proof of confluence of the rewriting system $(\Sigma_\calH, R_\calH)$ associated to the contextual family of associahedra $\calH$ provides a proof of Mac Lane's coherence theorem, with the pentagon in $\calH(2)$ acting as the coherence condition. 
The following examples explain the translation between the language of hypergraph polytopes and the language of monoidal categories. 

\begin{example}
Consider the linear tree
\vspace{-.7cm}
\begin{center}
$$\xymatrix @-1.65pc {{\cal T} & \eqdef  &X \ar @{-}[rr]^{1}&& Y \ar @{-}[rr]^{2}&& Z \ar @{-}[rr]^{3}&& U}.
  $$
\end{center}
Then $\hyper{L}({\cal T})$ is the associahedron $\hyper{K}^3$. 
The constructs of ${\cal T}$ decorate a pentagon as follows

\begin{center}
$$\xymatrix @-1.65pc {&& 3(2(1)) \ar @{->}[ddll]_{3(\set{1,2})} \ar @{->}[dddrr]^{\set{2,3}(1)}&& \\
 &&&&\\
3(1(2))   \ar @{->}[dd]^{\set{1,3}(2)}&&   && \\
 &&&&2(1,3) \ar @{->}[dddll]^{\set{1,2}(3)}\\
 1(3(2)) \ar @{->}[ddrr]^{1(\set{2,3})} &  &\\
 &&&&\\
 &&1(2(3))&&}$$
\end{center}
and are in bijective correspondence with the vertices and edges of Mac Lane's pentagon. 
The encoding is given as follows:
  \begin{itemize}
  \item $(X\otimes_1 Y)\otimes_2 (Z\otimes_3 U)$, where we annotated the ``compositions'' $\otimes$ with the vertices of $\hyper{K}^2$, can be written $\otimes_2(\otimes_1(X,Y),\otimes_3(Z,U))$ in prefix (or tree) notation. Then we get 
 $2(1,3)$ by removing the leaf nodes of that tree.
 \item $\alpha_{X,Y,Z}\otimes_3 U$ can be interpreted as $(X\otimes_1 Y\otimes_2 Z)\otimes_3 U$ (a non fully parenthesized expression), which likewise 
 translates as $3(\set{1,2})$,where $3(-)$ makes the job of contextualization.
 \item Likewise, we can move from
$\alpha_{X,Y\otimes_2 Z,U}$ to $X\otimes_1(Y\otimes_2 Z)\otimes_3 U$ to $\set{1,3}(2)$, where $2$ makes the job of instantiation.
\end{itemize}
\end{example}

\begin{example} 
  \label{Huet-MacLane}
Taking the 4-dimensional associahedron $\hyper{K}^{\set{0,1,2,3,4}}$, we get the following instance in context of $\hyper{K}^3=\recrestr{\hyper{K}^{\set{0,1,2,3,4}}}{\set{1,2,3}}$, i.e., of Mac Lane's coherence condition:
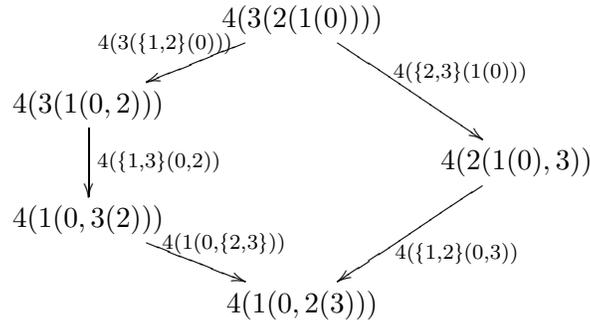
\begin{center}
$$\xymatrix @-1.65pc {&& 4(3(2(1(0)))) \ar @{->}[ddll]_{4(3(\set{1,2}(0)))} \ar @{->}[dddrr]^{4(\set{2,3}(1(0)))}&& \\
 &&&&\\
4(3(1(0,2)))  \ar @{->}[dd]^{4(\set{1,3}(0,2))}&&   && \\
 &&&&4(2(1(0),3)) \ar @{->}[dddll]^{4(\set{1,2}(0,3))}\\
 4(1(0,3(2))) \ar @{->}[ddrr]^{4(1(0,\set{2,3}))} &  &\\
 &&&&\\
 &&4(1(0,2(3)))&&}$$
\end{center}
We recover the (encoding of the) edge 
 $$
 \xymatrix @-2pc {&& ((((X_1\otimes_0 X_2)\otimes_1 Y)\otimes_2 Z)\otimes_3 U)\otimes_4 V \ar @{->}[dddddll]_(.6){(\alpha_{(X_1\otimes X_2),Y,Z}\otimes U)\otimes V\quad} && \\
 &&&&\\
 &&&&\\
  &&&&\\
    &&&&\\
( ((X_1\otimes_0 X_2)\otimes_1 (Y\otimes_2 Z))\otimes_3 U)\otimes_4 V  &&   && \\
\ }
$$
as the top left edge above.
\end{example}

Here, the fact that the family of associahedra is contextual implies in particular that the local confluence diagram associated to the expression
$$((X_1\otimes_0 X_2)\otimes_1 Y\otimes_2 Z\otimes_3 U)\otimes_4 V $$
which takes place on the $4$-dimensional associahedron, has the same shape as the critical confluence diagram associated to the expression
$$ (X \otimes_1 Y \otimes_2 Z \otimes_3 U)$$
and thus that the former one can be seen as an instance in context of MacLane's pentagon.  
Rigorously, the above data define a term rewriting system, that we shall call the Huet--Mac Lane rewriting system, on a signature consisting of a single operation $\otimes$ and on linear terms written with six variables  $X_1,X_2,Y,Z,U,V$. 
This rewriting system is in exact correspondence with our rewriting system for  $\hyper{K}^{\set{0,1,2,3,4}}$. 
More precisely, we can go from constructs to  Huet--Mac Lane terms by applying the following recipe. 
Consider the linear tree 
\vspace{-.7cm}
\begin{center}
$$\xymatrix @-1.65pc {{\cal T}' & \eqdef  &X_1 \ar @{-}[rr]^{0}&& X_2 \ar @{-}[rr]^{1}&& Y \ar @{-}[rr]^{2}&& Z \ar @{-}[rr]^{3}&& U \ar @{-}[rr]^{4}&& V}.
  $$
  \end{center}
Then, building, say the construct $4(3(1(0,2)))$ ``on this tree'' rather than on its associated  line graph $\mathbb{L}({\cal T}')=\hyper{K}^5$, we can see that picking 4 amounts to cutting $\cal T'$ by removing edge 4. This leaves  the one-node tree $V$ alone on the right and the term associated to $3(1(0,2))$ on the left, thus determining $-\otimes V$. And so on, until reaching $( ((X_1\otimes X_2)\otimes  (Y\otimes Z))\otimes U)\otimes V $. 
Note that the information ``leaving $V$ alone" is lost on the associated line graph. 

We note that linear Huet--Mac Lane terms, as well as the  partially parenthesized expressions such as those written above, can also be described as all possible nestings on, say~${\cal T'}$.

As for the ``easy'' local confluence diagrams of type (a1) and (a2), let us point out that they correspond in categorical terms to the bifunctoriality and naturality conditions, respectively, as exemplified below:
\begin{center}
\begin{tikzpicture}
\node (a) at (0,1) {};
\node (b) at (-2,0) {};
\node (c) at (2,0) {};
\draw[->] (a)--(b);
\draw[->] (a)--(c);
\node at (0,1.3) {$((X\otimes Y)\otimes Z)\otimes((U\otimes V)\otimes W) $};
\node at (-3,-0.3) {$(X\otimes (Y\otimes Z))\otimes((U\otimes V)\otimes W) $};
\node at (3,-0.3) {$ ((X\otimes Y)\otimes Z)\otimes(U\otimes (V\otimes W))$};
\end{tikzpicture}
\end{center}

\medskip

\begin{center}
  \begin{tikzpicture}
  \node (a) at (0,1) {};
  \node (b) at (-2,0) {};
  \node (c) at (2,0) {};
  \draw[->] (a)--(b);
  \draw[->] (a)--(c);
  \node at (0,1.3) {$(((X\otimes Y)\otimes Z)\otimes U)\otimes V $};
  \node at (-3,-0.3) {$ ((X\otimes (Y\otimes Z))\otimes U)\otimes V$};
  \node at (3,-0.3) {$ ((X\otimes Y)\otimes Z)\otimes (U\otimes V)$};
  \end{tikzpicture}
\end{center}

As we have seen in \cref{non-contextual-2}, this interpretation would not hold anymore if one were to consider the cycle graph instead of the line graph, that is, if one were to identify $X$ and $V$ in the expressions above. 

\medskip
All in all, we can summarize the orginal Huet correspondence and its factorisation through associahedra as follows:
\begin{itemize}
\item[(0)] iterated tensor functors are viewed as linear terms,
\item[(1)] associators between such functors are viewed as rewriting steps,
\item[(2)] Mac Lane's pentagons are viewed as local confluence diagrams for overlapping pairs,
\end{itemize}
corresponding to 0-faces, 1-faces and pentagonal 2-faces of associahedra, respectively.


\subsubsection{Operahedra}
In the case of operahedra, it turns out that the shape of a $\set{x_1,x_2,x_3}$-face is entirely determined by the relative positions of the edges $x_1,x_2,x_3$ in the underlying  planar tree.  Let us make this more precise.
Let $S$ be an $X$-face of $\mathbb{L}({\cal T})$, for some  planar tree $\cal T$. Then its shape is given by (the linear graph associated to)  the tree ${\cal T}_{\setminus X}$ obtained from  ${\cal T}$  by contracting all edges except the three elements of $X$ (which are edges of ${\cal T}$).  We leave the details to the reader, but note that
this illustrates contextuality: this shape is only determined by ${\cal T}$ and $X$, and does not depend on the position of $X$ in $S$.

Let us also say a word on the actual coherence statement for categorified operads: it relies on a signature and rewriting rules much in the spirit of the Huet--Mac Lane rewriting system (see~\cite{DP15,laplante-anfossiDiagonalOperahedra2022a,CLA1}), that correspond to constructs of operahedra, to their equivalent representations as nestings of planar trees (the linear trees of associahedra being a particular case), and to our rewriting systems on the linear graphs of planar trees. 


\subsubsection{Permutahedra and friends}
It seems likely that the family of permutahedra admits a similar categorical coherence theorem. 
The corresponding algebraic structure would be here that of \emph{permutads} \cite{LodayRonco11,Markl19}. 
In the same fashion as associahedra are operahedra associated to linear trees, permutahedra are operahedra associated to $2$-leveled trees \cite{laplante-anfossiDiagonalOperahedra2022a}.
Therefore, one could define categorified permutads by adapting the definition of categorified operads \cite{CLA1} to these trees.

In the case of contextual graph-associahedra, it seems likely that a corresponding coherence theorem could be associated to a certain type of categorified reconnectads \cite{DotsenkoKeilthyLyskov}.
The situation is summarized in \cref{table:contextual-hyper}.

\begin{table}[h!]
	\begin{center}
	\begin{tabular}{c|c|c}
	Family & Algebraic structure & Coherence theorem \\
	\hline
	Simplices & - & - \\
	Cubes & - & - \\
	Associahedra & Monoidal category & \cite{MacLane63} \\
	Permutahedra & Categorified permutads & - \\
	Operahedra & Categorified operads & \cite{DP15,CLA1} \\
	Contextual graph-associahedra & Categorified reconnectads & - \\
	Contextual nestohedra & - & - 
	\end{tabular}
	\end{center}
  \caption{Families of contextual hypergraphs, the categorical structures that they (conjecturally) encode, and their associated coherence theorems.}
  \label{table:contextual-hyper}
\end{table}

\bigskip 

\bibliographystyle{amsalpha}

\bibliography{Coherence2}

\end{document}